\documentclass[sn-mathphys-num]{sn-jnl}

\usepackage{graphicx}%
\usepackage{multirow}%
\usepackage{amsmath,amssymb,amsfonts}%
\usepackage{amsthm}%
\usepackage{mathrsfs}%
\usepackage[title]{appendix}%
\usepackage{xcolor}%
\usepackage{textcomp}%
\usepackage{manyfoot}%
\usepackage{booktabs}%
\usepackage{algorithm}%
\usepackage{algorithmicx}%
\usepackage{algpseudocode}%
\usepackage{listings}%

\usepackage{mathrsfs}%
\usepackage{hyperref}%

\theoremstyle{plain}
\newtheorem{theo}{Theorem}[section]
\newtheorem{prop}{Proposition}[section]
\newtheorem{coro}{Corollary}[section]
\newtheorem{lemm}{Lemma}[section]
\newtheorem{fact}{Fact}[section]

\newtheorem{main}{Main Theorem}
\newtheorem{ques}{Question}[section]
\newtheorem{prob}{Problem}[section]

\theoremstyle{definition}
\newtheorem{defi}{Definition}[section]
\newtheorem{exam}{Example}[section]
\newtheorem{rema}{Remark}[section]

\begin{document}

\title[Limit theorems for the total scalar curvature]{Limit theorems for the total scalar curvature}


\author*[1]{\fnm{Shota} \sur{Hamanaka}\, 
\email{hamanaka1311558@gmail.com}}



\affil*[1]{\orgdiv{Department of Mathematics}, \orgname{Graduate School of Science, Osaka University}, \orgaddress{\street{1-1 Machikaneyama-cho}, \city{Toyonaka}, \postcode{560-0043}, \state{Osaka}, \country{Japan}}}




\abstract{ We study some preservation phenomena for lower bound of total scalar curvatures on a smooth manifold.
In particular, we prove that the lower bound of the weighted total scalar curvature (which is known as Perelman's $\mathcal{F}$-functional) on a closed $n$-manifold is preserved 
under the $W^{1, p}~(p > n^{2}/2)$-convergence of Riemannian metrics
and uniformly $C^{0}$-convergence of potential functions,
provided that each scalar curvature is nonnegative.
In the proof, we used a certain stability of the Ricci flow and the heat flow with the Ricci flow background.
We also give some examples that may provide clues to identify the weakest topology for such a preservation phenomenon of the lower bound.}




\maketitle

\section{Introduction}\label{sec1}

Gromov \cite{gromov2014dirac} proved the following ``$C^{0}$-limit theorem''.
\begin{theo}[{\cite[p.1118]{gromov2014dirac}} and \cite{bamler2016proof}]
\label{gromov-limit}
Let $M$ be a (possibly open) smooth manifold and 
$\kappa : M \rightarrow \mathbb{R}$ a continuous function.
Consider a sequence of $C^{2}$-Riemannian metrics $g_{i}$ on $M$
that converges to a $C^{2}$-Riemannian metric $g$ in the local $C^{0}$-sense. Assume that for all $i = 1,2, \cdots$
the scalar curvature $R(g_{i})$ of $g_{i}$ satisfies $R(g_{i}) \ge \kappa$
everywhere on $M.$
Then $R(g) \ge \kappa$ everywhere on $M.$
\end{theo}

In contrast, let $M$ be the same as in the above theorem, for given a continuous function $\sigma : M \rightarrow \mathbb{R},$
the set $\{ g \in \mathcal{M} |~R(g) \le \sigma~\mathrm{on}~M \}$ is $C^{0}$-dense in the set $\mathcal{M}$ of all smooth Riemannian metrics on $M$
({\cite[Theorem B]{lohkamp1995curvature}}).
On the other hand, in our forthcoming paper \cite{hamanaka2023upper}, we will show that in a fixed conformal class, the upper bound of the total scalar curvature is preserved under $C^{0}$-convergence of metric tensors under some assumptions.
Gromov \cite{gromov2014dirac} proved the above theorem by using a gluing technique and the resolution of Geroch's conjecture.
Later, Bamler \cite{bamler2016proof} gave an alternative proof of this theorem using the Ricci--DeTurck flow.
On the other hand, Lee--Topping \cite{lee2022metric} recently proved that nonnegativity of scalar curvature is not preserved in dimension at least four under the uniform convergence of Riemannian distances.
More recently, Kazaras--Xu \cite{kazaras2025codimension} have shown that such a phenomenon can also occur in dimension three.
For other studies on the behavior of the scalar curvature lower bound under various weak topologies, see, for example, \cite{burkhardt2019pointwise,huang2023scalar, jiang2023weak, kazaras2023drawstrings, kazaras2025codimension}.

As we will see below, we can observe that the point-wise version of Gromov's theorem (Theorem \ref{gromov-limit}) is false in general.
That is, we can easily construct an example of $C^{2}$-Riemannian metrics $(g_{i})$ on a smooth manifold $M$ which satisfies the following:
$g_{i}$ converges to a $C^{2}$-Riemannian metric $g$ on $M$ in the $C^{0}$-sense as $i \rightarrow \infty.$
And there is a point $p \in M$ such that for some $\kappa \in \mathbb{R},$
\[
R(g_{i})(p) \ge \kappa~~\mathrm{for~all}~i \in \mathbb{N},
\]
but
\[
R(g)(p) < \kappa.
\]

\noindent
Indeed, Lohkamp gave an example in {\cite[Lecture~Series~2,~Counterexample~2.3.2]{besson1996riemannian}}.
\begin{exam}[{\cite[Lecture~Series~2,~Counterexample~2.3.2]{besson1996riemannian}}. cf. Example \ref{exam-C10} and \ref{exam-closed} in this paper]
\label{exam-below}
For each $i \in \mathbb{N},$ we define a smooth metric on $\mathbb{R}^{n}$ as 
\[
g_{i} :=
\begin{cases}
\exp{ \left(2 f_{i} \right)} \cdot g_{Eucl} & \mathrm{on}~D_{\alpha, i} := \{ x \in \mathbb{R}^{n}~|~ |x|_{g_{Eucl}}^{2} \le \alpha/i \}, \\
g_{Eucl} & \mathrm{in}~\mathbb{R}^{n} \setminus D_{2\alpha, i}.
\end{cases}
\]
Here, $g_{Eucl}$ denotes the Euclidean metric on $\mathbb{R}^{n},$
$|x|_{g_{Eucl}}$ is the Euclidean distance from $x$ to the origin $o,$ the smooth function
\[
f_{i} := \frac{\alpha}{i} - x_{1}^{2} -x_{2}^{2} - \cdots - x_{n}^{2}
\]
whose support is contained in $D_{2\alpha, i},$ and $\alpha \in \mathbb{R}$ is a positive constant.
Then, using the following fact, we have
$R (g_{i}) (o) = C(n) > 0.$
Moreover,
$g_{i}$ converges to $g_{Eucl}$
uniformly in the $C^{0}$-sense on $\mathbb{R}^{n}.$
\begin{fact}[Conformal change of the scalar curvature]
For a Riemannian metric $g$ and a $C^{2}$-function $\phi,$
set $\Bar{g} := e^{2 \phi} g,$ then
\[
R (\Bar{g}) = e^{-2 \phi} R(g) - 2 (n-1)~e^{-2 \phi} \Delta_{g} \phi-  (n-2) (n-1)~e^{-2 \phi}~|d \phi|_{g}^{2}.
\]
\end{fact}

\noindent
The proof of this fact is a straightforward calculation.
Note that the scalar curvature lower bounds are guaranteed only at the origin $o$ in this example. But we can never take a metric whose scalar curvature is bounded from below by some positive constant on a small neighborhood of $o$ and nonnegative on the whole manifold $\mathbb{R}^{n},$ and which is equal to the Euclidean metric outside a compact subset of $\mathbb{R}^{n}.$
Indeed, if such a metric exists, then we can construct a metric on the $n$-dimensional torus whose scalar curvature is nonnegative everywhere and positive somewhere. 
But, this is impossible by the resolution of Geroch's conjecture (see \cite{gromov1980spin, schoen1979existence, schoen1979structure}).
Of course, we can apply Theorem \ref{gromov-limit} to this sequence $(g_{i})$ and $M = \mathbb{R}^{n}.$
But, we can only take the lower bound $\kappa$ such that $\sup_{\mathbb{R}^{n}} \kappa \le 0$ in this case because the support of $f_{i}$ shrinks as $i \rightarrow \infty.$
Hence, we can only obtain the trivial fact $R(g_{Eucl})  = 0 \ge \kappa$ even if we use Theorem \ref{gromov-limit}.
\end{exam}

In this paper, we investigate some total scalar curvature versions of Theorem \ref{gromov-limit}.
More precisely, we will consider the following problems:
\begin{prob}\label{prob-1}
    Let $M$ be a (possibly non-compact) smooth manifold,
$\{ g_{i} \}_{i \in \mathbb{N}}$ a sequence of $C^{2}$-Riemannian metrics
and $g$ a $C^{2}$-Riemannian metric on $M.$
If
$g_{i}$ converges to $g$
in some sense (with respect to $g$)
and
\[
\int_{M} R(g_{i})\, d\mathrm{vol}_{g_{i}} \ge \kappa~~~~\mathrm{for~all}~~i
\]
for some constant $\kappa \in \mathbb{R}.$
Here $R_{g_{i}},\, d\mathrm{vol}_{g_{i}}$ denote respectively the scalar curvature of $g_{i}$ and the Riemannian volume measure of $g_{i}.$
Then, does 
\[
\int_{M} R(g)\, d\mathrm{vol}_{g} \ge \kappa
\]
hold?
\end{prob}

Or, more generally, 
\begin{prob}\label{prob-2}
Let $a, b, c \in \mathbb{R}$ be given constants.
    Let us consider a sequence of $C^{2}$ metrics $g_{i}$ on $M$ that converge to a $C^{2}$ metric in some sense and a sequence of functions $f_{i}$ on $M$ that converges to a function $f$ in some sense.
    Assume that they satisfy
\[
\int_{M} R^{a,b,c}_{f_{i}}(g_{i}) e^{-f_{i}}\, d\mathrm{vol}_{g_{i}}
:= \int_{M} (R_{g_{i}} + a |\nabla f_{i}|_{g_{i}}^{2} + be^{cf}) e^{-f_{i}}\, d\mathrm{vol}_{g_{i}}
\ge \kappa~~~~\mathrm{for~all}~~i
\]
for some constant $\kappa \in \mathbb{R}.$
Then, we ask whether
\begin{equation}\label{eq-prob-2}
\int_{M} R^{a,b,c}_{f}(g) e^{-f}\, d\mathrm{vol}_{g} \ge \kappa
\end{equation}
holds or not.
\end{prob}
\begin{rema}
\begin{itemize}
    \item The left hand side of (\ref{eq-prob-2}) when $(a,b,c) = (1, 0, 0)$, is well-known as Perelman's $\mathcal{F}$-functional \cite{perelman2002entropy}.
    Moreover, if $f$ is $C^{2}$ (more generally, if $\Delta_{g} f$ exists in a certain weak sense), then it can be rewritten as 
    \begin{equation}\label{eq-p-scalar}
    \int_{M} (R_{g} + 2\Delta_{g} f - |\nabla f|_{g}^{2}) e^{-f}\, d\mathrm{vol}_{g}.
    \end{equation}
    It is also known that in a certain sense, the gradient flow for the $\mathcal{F}$-functional under the constraint that $e^{-f}\, d\mathrm{vol}_{g}$ is fixed, is the Ricci flow.
    And, the integrand of (\ref{eq-p-scalar}):
    \[
R_{g} + 2\Delta_{g} f - |\nabla f|_{g}^{2}
\]
    is known as the \textit{P-scalar curvature}.
    Interestingly, it has been found that the scalar curvature and the P-scalar curvature share several analogous properties (cf. \cite{abedin2017p,baldauf2022spinors, chu2024non, law2025positive, wang2025rigidity}).
    \item When $f$ is $C^{2}$ and $(a,b,c) = \left( \frac{m-1}{m}, m(m-1)\mu, \frac{2}{m} \right)$, the weighted total scalar curvature in the right side of (\ref{eq-prob-2}) coincides with the one considered in \cite{case2019weighted}
    via integration by parts.
    It is known that such an analogue of the scalar curvature plays the role of the scalar curvature in various geometric problems (cf. \cite{case2015yamabe, case2019weighted, lott2007remark, perelman2002entropy}).
    \end{itemize}
\end{rema}
\medskip
We emphasize that we will only consider the situation in which metrics converge to some metric with respect to a certain topology which is weaker than $C^{2},$ 
each metric is at least $C^{2},$ and the underlying manifolds we consider are assumed to be smooth.

\medskip
First of all, Problem \ref{prob-1} is true with $C^{0} \cap W^{1,2}$-convergence on a closed manifold $M$ as follows.
\begin{fact}
\label{theo-3}
 Let $M^{n}$ be a closed manifold of dimension $n \ge 3$ and $g$ a $C^{2}$ Riemannian metric on $M.$
Assume that $(g_{i})$ is a sequence of $C^{2}$ Riemannian metrics on $M$ such that 
$g_{i}$ converges to $g$ on $M$ in the $C^{0} \cap W^{1, 2}$-sense as $i \rightarrow \infty$. 
Then, 
\begin{equation}\label{eq-total-conv}
\int_{M} R_{g_{i}}\, d\mathrm{vol}_{g_{i}} \rightarrow \int_{M} R_{g}\, d\mathrm{vol}_{g}
\end{equation}
as $i \rightarrow \infty$.
In particular, Problem \ref{prob-1} holds true with respect to this topology of convergence.
\end{fact}
Here, we said that a sequence of metrics $(g_{i})_{i}$ converges to a metric $g$ on $M$ in the $W^{1, p}$-sense if $g_{i}$ and all first weak derivatives of it respectively converge to those of $g$ with respect to the $L^{p}$-norm of $g.$ (Since $M$ is compact, if $(g_{i})$ converges in the $W^{1, p}$-sense with respect to $g,$ then it also converges $W^{1, p}$-sense with respect to any fixed reference metric on $M.$)
\begin{proof}[Proof of Fact \ref{theo-3}]
    This fact follows from the integration by parts formula (\ref{eq-distr-def}) in Section \ref{section-coro} with $u \equiv 1$ and $h = g$.
\end{proof}
\begin{rema}\label{rema-counterexample}
Note that (\ref{eq-total-conv}) also implies the upper bound of the total scalar curvatures is preserved under the same topology of convergence. 
    On the other hand, \cite{lohkamp1995curvature} (see also \cite{hamanaka2023upper}) suggests that in the above statement, if ``$C^{0} \cap W^{1,2}$'' is replaced by ``$C^{0}$'', (\ref{eq-total-conv}) does not hold in general.
    However, it is not clear whether Problem \ref{prob-1} holds with $C^{0}$-topology or not.
    Nevertheless, Example \ref{exam-closed} in this paper still suggests that it does not hold in general.
    Moreover, in a fixed conformal class, the upper bound of the total scalar curvature is preserved under the $C^{0}$-convergence (see \cite{hamanaka2023upper}).
\end{rema}
\begin{coro}
Let $\mathcal{M}$ be the space of all $C^{2}$-Riemannian metrics on a closed manifold $M.$
For any constant $\kappa \in \mathbb{R},$ the subspace
\[
\left\{ g \in \mathcal{M}~\middle|~\int_{M} R(g)\, d\mathrm{vol}_{g} \ge \kappa \right\}~(\subset \mathcal{M})
\]
is closed in $\mathcal{M}$ with respect to the $C^{0} \cap W^{1, 2}$-topology.
\end{coro}
\medskip
If $M^{2}$ is closed (i.e., compact without boundary) surface, from the Gauss-Bonnet theorem,
\[
\int_{M^{2}} R(g)\, d\mathrm{vol}_{g} = 4 \pi \chi(M)
\]
for each Riemannian metric $g$ on $M.$
Here $\chi(M)$ denotes the Euler characteristic of $M.$
Hence it is sufficient to consider the above problem (unweighted case) in dimension $n \ge 3$
and, unless otherwise mentioned, we will assume below that the dimensions of manifolds are greater than or equal to three.

On the other hand, if $M^{2n}$ is a closed complex $n~(\mathrm{real}~2n)$-manifold ($n \ge 1$) and $g, g_{i}~(i =1,2, \cdots )$ are K{\"a}hler metrics on $M.$
Let $\omega$ and $\omega_{i}$ be the K{\"a}hler forms of $g$ and $g_{i}$ respectively.
Assume that $g_{i}$ converges to $g$ (hence $\omega_{i}$ converges to $\omega$) in the $C^{0}$-sense as $i \rightarrow \infty.$
Then 
{\small \[
\begin{split}
\int_{M} R(g_{i})\, d\mathrm{vol}_{g_{i}} &= \frac{4 \pi}{(n-1)!} (c_{1}(M) \smallsmile [\omega_{i}^{n-1}](M)) \\
&\longrightarrow \frac{4 \pi}{(n-1)!} (c_{1}(M) \smallsmile [\omega^{n-1}](M)) = \int_{M} R(g)\, d\mathrm{vol}_{g}~~~(i \rightarrow \infty),
\end{split}
\]}   
where $c_{1}(M)$ denotes the first Chern class of $M.$
Note that we assumed here that the limiting metric $g$ is K{\"a}hler metric on $M$ as well.
In contrast to this, in our main theorem \ref{theo-1} below, we will not assume that the limiting metric $g$ is a Ricci soliton.
Although it deviates a bit from our subject, the following interesting result about lower bounds of scalar curvature integrals is also known.
\begin{theo}[{\cite[Lecture~Series~1,~Theorem~4.1]{besson1996riemannian}}]
\label{theo-hyp}
Let $M$ be a compact $n$-dimensional manifold $(n \ge 3)$ carrying a hyperbolic metric $g_{0}.$
There is a neighborhood $\mathcal{U}$ of $g_{0}$ in the space of all Riemannian metrics with the $C^{2}$-topology such that for any $g \in \mathcal{U},$
\[
\int_{M} \left( R(g)_{-} \right)^{n/2} d\mathrm{vol}_{g} \ge \int_{M} \left( R(g_{0})_{-} \right)^{n/2}\, d\mathrm{vol}_{g_{0}}
\]
and equality if and only if $g$ is isometric to $g_{0}.$
Here, $R(g)_{-} := \max \left\{ - R(g),~0 \right\}.$
\end{theo}

As our main result in this paper (related to Problem \ref{prob-2}), we will prove the following theorem.

\begin{main}
\label{theo-2}
Let $p > n^{2}/2$ and $a,b,c \in \mathbb{R}$.
Suppose that $M^{n}$ is a closed $n$-manifold ($n \ge 2$), $g$ is a $C^{2}$ Riemannian metric on $M,$ and $(g_{i})$ is a sequence of $C^{2}$ Riemannian metrics on $M$ such that 
$g_{i}$ converges to $g$ on $M$ in the $W^{1, p}$-sense as $i \rightarrow \infty.$
Let $(f_{i})$ be a family of functions on $M$ and $f$ a function on $M.$
Assume the following:
\begin{itemize}
\item[(1)] There is a positive constant $\Lambda > 0$ such that $f$ and $f_{i}~(i \ge 0)$ are $\Lambda$-Lipschitz functions on $M,$
\item[(2)] $f_{i} \overset{C^{0}}{\longrightarrow} f$ uniformly on $M,$
\item[(3)] $R(g_{i}) \ge 0$ on $M$ for all $i,$
\item[(4)] $\int_{M} R^{a,b,c}_{f_{i}}(g_{i})\, e^{-f_{i}} d\mathrm{vol}_{g_{i}} \ge \kappa~~(\kappa \in \mathbb{R})$, where $\int_{M} R^{a,b,c}_{f_{i}}(g_{i})\, e^{-f_{i}} d\mathrm{vol}_{g_{i}}$ is the weighted total scalar curvature defined as in Problem \ref{prob-2} above.
\end{itemize}
Then
\[
\int_{M} R^{a,b,c}_{f}(g)\, e^{-f} d\mathrm{vol}_{g} \ge \kappa.
\]
That is, Problem \ref{prob-2} holds true with $W^{1, p}~(p > n^{2}/2)$-convergence of metrics and $C^{0}$-uniformly convergence of $\Lambda$-Lipschitz functions.
\end{main}
\begin{rema}
\begin{itemize}
  \item For the case of $\kappa = 0$, $f_{i} \equiv f \equiv 0$, this claim follows from Theorem \ref{gromov-limit}, Fact \ref{theo-3} respectively.
\item The assumption $p > n^{2}/2$ is necessary for the technical reason in order to obtain Lemma \ref{lemm-sub}.
And, the assumption $R(g_{i}) \ge 0$ is necessary for the technical reason in order to obtain Lemma \ref{lemm-key}.
Nevertheless, we speculate the assumptions $R(g_{i}) \ge 0~(i = 1, 2, \cdots)$ in Main Theorem  \ref{theo-2} may not be needed.
\item This is non-trivial even in the two-dimensional case because $f_{i}$ is non-constant in general,
hence we cannot use the Gauss-Bonnet theorem.
\item For example, $(1)$ and $(2)$ are automatically satisfied in case that $e^{-f_{i}} d\mathrm{vol}_{g_{i}} = d\mathrm{vol}_{g_{0}} = e^{-f} d\mathrm{vol}_{g}$ for some $C^{0}$ metric $g_{0}$,
and $g_{i} \overset{C^{1}}{\longrightarrow} g.$
Indeed, since each $f_{i}$ can be locally written as $f_{i} = \log (\sqrt{\mathrm{det}\, g_{i}}) - \log (\sqrt{\mathrm{det}\, g_{0}})$ ($f$ is also represented in the same form) and $g_{i} \overset{C^{1}}{\longrightarrow} g,$
so the norm of the first derivatives $(|\nabla f_{i}|_{g})_{i}$ are uniformly bounded and $f_{i} \rightarrow f$ uniformly on $M.$
\end{itemize}
\end{rema}
From the Morrey embedding:
\[
 W^{1, p} \hookrightarrow C^{0, 1 - \frac{n}{p}} ~~~(p > n),
\]
we immediately obtain the following corollary of Theorem \ref{gromov-limit} and Main Theorem \ref{theo-2}.
\begin{coro}
Let $a,b,c \in \mathbb{R}$.
Let $\mathcal{M}$ be the space of all $C^{2}$-Riemannian metrics on a closed manifold $M$ and $\mathrm{Lip}(M)$ the space of all Lipschitz continuous functions on $M$.
Then, for each $p > n^{2}/2$, any nonnegative continuous function $\sigma \in C^{0}(M, \mathbb{R}_{\ge 0})$, positive constant $\Lambda > 0$ and constant $\kappa \in \mathbb{R},$ the subspace
\[
\mathcal{M}^{a,b,c}_{\sigma, \Lambda, \kappa} :=
\left\{ (g, f) \in \mathcal{M} \times \mathrm{Lip}(M)~\middle|~R_{g} \ge \sigma,f : \Lambda\mathrm{-Lipschitz},~\int_{M} R^{a,b,c}_{f}(g)e^{-f}\, d\mathrm{vol}_{g} \ge \kappa \right\}
\]
is closed in $\mathcal{M} \times \mathrm{Lip}(M)$ with respect to the product topology the $W^{1, p}$ and the $C^{0}$-topology.
\end{coro}
In exactly the same way as in the proof of Main Theorem \ref{theo-2} (more simply), one can also prove the following.
\begin{coro}\label{coro-weighted-scalar}
    The same statement of Main Theorem \ref{theo-2} as with replacing the integrands $R^{a,b,c}_{f_{i}}(g_{i}), R^{a,b,c}_{f}(g)$ with $R(g_{i}), R(g)$ respectively still holds. 
\end{coro}
As a corollary of this, we can obtain the following and from it, we can also make a definition of scalar curvature lower bounds in a weak sense.
\begin{coro}[$=$ Corollary \ref{coro-distr}]
\label{coro-distr-0}
  Let $p > n^{2}/2$ and $\kappa$ a constant.
  Suppose that $M$ is a closed manifold of dimension $n \ge 2,$ $g$ is a $C^{2}$ Riemannian metric on $M$ and $(g_{i})$ is a sequence of $C^{2}$ metric on $M.$
  Assume the following:
  \begin{itemize}
    \item[(1)] a sequence $(\phi_{i})$ of nonnegative continuous functions on $M$ satisfying:
    for any positive constant $a > 0$ there is a positive constant $\Lambda > 0$ such that $\log (\phi_{i} + a)$ is $\Lambda$-Lipschitz on $M$ for all $i,$
    \item[(2)] $(\phi_{i})$ converges to some nonnegative continuous function $\phi$ in the uniformly $C^{0}$-sense on $M,$
    \item[(3)] $R(g_{i}) \ge 0$ on $M$ for each $i,$
    \item[(4)] $\int_{M} R(g_{i}) \phi_{i}\, d\mathrm{vol}_{g_{i}} \ge \kappa \int_{M} \phi_{i}\, d\mathrm{vol}_{g_{i}},$
    \item[(5)] $g_{i}$ converges to $g$ in the $W^{1, p}$-sense.
  \end{itemize}
  Then,
  \[
  \int_{M} R(g) \phi\, d\mathrm{vol}_{g} \ge \kappa \int_{M} \phi\, d\mathrm{vol}_{g}.
             \]
\end{coro}
\begin{rema}
\label{rema-distr}
  From $(3), (5)$ and Remark \ref{rema-distr-smooth}, it is known that $\int_{M} R(g) \psi\, d\mathrm{vol}_{g} \ge 0$ for all \textbf{smooth} nonnegative function $\psi.$
  Hence it is reasonable to consider case that $\kappa \ge 0$ in this setting.
\end{rema}
We give the necessary notions and prove this corollary in Section \ref{section-coro}.
Based on this type of limit theorem, we can define a new generalized notion of scalar curvature lower bound via the existence of certain types of approximate sequences as follows.
\begin{defi}
\label{defi-new}
  Let $M^{n}$ be a smooth closed $n$-manifold and $\kappa$ a constant. For any $W^{1, p}~(p > n^{2}/2)$ metric $g$ on $M,$ $g$ is of $R(g) \ge \kappa$ in \textit{the approximate distributional sense} if
  for any nonnegative continuous function $\phi$ there is an approximate sequence $(\phi_{i})$ satisfying $(1), (2)$ in Corollary \ref{coro-distr-0}, and there exists a $W^{1, p}~(p > n^{2}/2)$-approximate sequence of $C^{2}$-metrics $(g_{i})$ satisfying $(3)$-$(5)$ in Corollary \ref{coro-distr-0}. 
\end{defi}
Suppose now that a metric $g$ in Definition \ref{defi-new} is actually $C^{2}.$
Then, from Corollary \ref{coro-distr-0}, $R(g) \ge \kappa$ in the approximate distributional sense on $M$ implies that the same bound $R(g) \ge \kappa$ holds in the distributional sense on $M.$
Then, as a result of Remark \ref{rema-distr-smooth}, $R(g) \ge \kappa$ in the conventional sense .
Note that $R(g) \ge 0$ in the conventional sense in this case due to $(3)$ of Corollary \ref{coro-distr-0} and Theorem \ref{gromov-limit}. 

\smallskip
As mentioned above, we cannot weaken the assumption ``$C^{0} \cap W^{1,2}$-convergence'' to ``$C^{0}$-convergence'' in Fact \ref{theo-3}.
On the other hand, if we assume that each metric is a Ricci soliton with an additional assumption, then we can obtain a similar statement under weaker $C^{0}$-convergence of metrics.
Our second main theorem is the following.
\begin{main}
\label{theo-1}
Let $M^{n}$ be a closed $n$-manifold and $g$ a $C^{2}$ Riemannian metric on $M.$
Let $(g_{i})$ be a sequence of Ricci solitons on $M$ (i.e., $-2\, \mathrm{Ric}(g_{i}) = \mathcal{L}_{Y_{i}} g_{i} - 2\lambda_{i}\, g_{i}$ for some constant $\lambda_{i} \in \mathbb{R}$ and a vector field $Y_{i} \in \Gamma(TM)$) such that 
$g_{i}$ converges to $g$ on $M$ in the $C^{0}$-sense as $i \rightarrow \infty.$
Assume
\[
\int_{M} R(g_{i})\, d\mathrm{vol}_{g_{i}} \ge \kappa~~~\mathrm{for~some~constant}~\kappa \in \mathbb{R}.
\]
Moreover, assume that $\lambda_{i} \le C_{+}$ for all $i$ and some constant $C_{+} \in \mathbb{R}$ if $\kappa \ge 0$ (resp. $\lambda_{i} \ge C_{-}$ for some $C_{-} \in \mathbb{R}$ if $\kappa < 0$).
Then
\[
\int_{M} R(g)\, d\mathrm{vol}_{g} \ge \kappa.
\]
\end{main}
On a complete manifold, every Ricci soliton is a self-similar solution of the Ricci flow equation, and vice versa.
This self-similarity is one of the reasons why the assumption of convergence in Main Theorem \ref{theo-1} can be weaker than $W^{1, p}.$ 

\medskip
The rest of the paper will be arranged as follows.
In Section \ref{section-preliminaries}, we will prepare some lemmas to prove our main theorems in the next section.
In Section \ref{section-proof}, we will prove our main theorems.
In Section \ref{section-coro}, we will prove some corollaries of our main theorems.
In Section \ref{section-examples}, we will give several examples that suggest that the assumptions of regularity or compactness of manifolds in Fact \ref{theo-3} and our main theorems are optimal in some sense.

\section{Preliminaries}
\label{section-preliminaries}
First, we need the following stability result for the Ricci-DeTurck flow like what is proved in {\cite[Lemma~2]{bamler2016proof}}.
More precisely, we need the following.
\begin{lemm}[\cite{simon2002deformation, koch2015parabolic}]
\label{lemm-1}
Let $(M, h)$ be a closed Riemannian manifold endowed with a $C^{2}$-Riemannian metric $h.$
Then, there are constants $\tau, \varepsilon > 0,~C < \infty$
such that the following is true:
Consider a $C^{2}$-Riemannian metric $g$ that is 
$(1+ \epsilon)$-bilipschitz close to $h.$
Then there is a continuous family of Riemannian metrics $(g_{t})_{t \in [0,\tau]}$ on $M$ such that the following holds:
\begin{itemize}
\item[$(a)$] For all $t \in [0, \tau],$ the metric $g_{t}$ is $1.1$-bilipschitz to $h.$

\item[$(b)$] $(g_{t})$ is smooth on $M \times (0, \tau]$ and the map $[0, \tau] \rightarrow C^{2}(M, S^{2}M),~t \mapsto g_{t}$ is continuous.
In particular, $t \mapsto R_{g_{t}}$ is continuous on $[0, \tau]$.

\item[$(c)$] $g_{0} = g$ and $(g_{t})_{t \in [0, \tau]}$ is a solution to the \textit{Ricci~DeTurck~flow~equation}
\begin{equation}\label{eq-RDF}
\frac{\partial}{\partial t} g_{t} = -2~\mathrm{Ric} (g_{t}) - \mathcal{L}_{X_{h}(g_{t})} g_{t},
\end{equation}
where $\mathcal{L}_{X_{h}(g_{t})} g_{t}$ denotes the Lie derivative of $g_{t}$ with respect to the time-dependent vector field $X_{h}(g_{t})$ defined in Remark \ref{rema-lemm-1} below.

\item[$(d)$] For any $t  \in (0, \tau]$ and any $m = 0, 1, 2$ we have
\[
|\nabla_{h}^{m} g_{t}|_{h} < \frac{C}{t^{m/2}}
\]
where $|\nabla_{h}^{m} g_{t}|_{h}$ denotes the norm with respect to $h$ of the covariant derivatives of $g(t)$ by the Levi-Civita connection of $h.$ 

\item[$(e)$] If $(g_{i, t})_{t \in [0, \tau]}$ is a sequence of solutions to (RDE) that are continuous on $M \times [0, \tau]$
and smooth on $M \times (0, \tau]$ and if $g_{i, 0}$
converges to some metric $g_{0}$ in the $C^{0}$-sense,
then there is a subsequence of $(g_{i, t})$ that converges to $(g_{t})$ in the $C^{0}$-sense on $M \times [0, \tau]$ and in the locally smooth sense on $M \times (0, \tau]$ with respect to $h.$
\end{itemize}
\end{lemm}

\begin{proof}
The items $(a), (c)$ and $(d)$ are the result of {\cite[Theorem~1.1]{simon2002deformation}}.
$(e)$ is the result of {\cite[Corollary 3.4]{burkhardt2019pointwise}}.
Differentiate the equation $(66)$ in \cite{shi1989deforming} twice with respect to $\nabla_{h}$, then it satisfies a parabolic equation treated in \cite{koch2015parabolic} (see also the last part of {\cite[Proof of Lemma 2]{bamler2016proof}}).
Then, $(b)$ follows from {\cite[Theorem 4.2]{koch2015parabolic}}.
\end{proof}

Thanks to {\cite[Theorem~3.11]{jiang2023weak}}, if we assume that $g_{i}$ converges to $g$ in the $W^{1, p}~(p > n)$-sense, we also obtain the following.
\begin{lemm}[cf. {\cite[Theorem~4.3]{simon2002deformation}}]
\label{lemm-basic}
Let $(M, g)$ be a closed Riemannian manifold endowed with a $C^{2}$-Riemannian metric $g.$
Then there are constants $\varepsilon,~\tau > 0$ and $C < \infty$
such that the following is true:
Consider a $C^{2}$-Riemannian metric $g'$ on $M$ that is 
$\varepsilon$-close to $g$ in the $W^{1, p}~(p > n)$-sense with respect to $g$
(i.e., $|g - g'|_{W^{1, p}(M, g)} < \varepsilon,$ where $|g - g'|_{W^{1, p}(g)}$ denotes the $W^{1, p}$-norm of the $(0, 2)$-tensor $g - g'$ with respect to $g$).
Then there is a continuous family of Riemannian metrics $(g'_{t})_{t \in [0,\tau]}$ on $M$ such that $(a)$-$(c)$ and $(e)$ in Lemma \ref{lemm-1} hold.
Moreover, the following $(d')$ holds instead of $(d)$ in Lemma \ref{lemm-1}.
\begin{itemize}
\item[$(d')$] For any $t \in (0, \tau]$ we have
\[
|\nabla_{g} g'_{t}|_{g} < \frac{C}{t^{n/2p}}~~\mathrm{and}~~|\nabla_{g}^{2} g'_{t}|_{g} < \frac{C}{t^{\frac{n}{4p} + \frac{3}{4}}}
\]
\end{itemize}
\end{lemm}
\begin{proof}
The statements except for the item $(d')$ are the results of Lemma \ref{lemm-1} $(a)$-$(c), (e)$ respectively.
The item $(d')$ is the result of {\cite[Theorem~3.11]{jiang2023weak}}.
Indeed, since $g'$ is $W^{1, p}$-close to $g,$ the $W^{1, p}$-bounded assumption in {\cite[Theorem~3.11]{jiang2023weak}}:
\[
\int_{M} |\nabla_{g} g'|^{p}\, d\mathrm{vol}_{g} \le A = A(\varepsilon, g)
\] 
is satisfied.
Therefore, we obtain the assertion from {\cite[Theorem~3.11]{jiang2023weak}}.
\end{proof}

\begin{rema}
\label{rema-lemm-1}
Let $(g_{t})_{t \in [0, T)}~(0 < T)$ be a solution of the Ricci-DeTurck flow equation with $g_{0} = g.$
Choose a background metric $\bar{g}$ on $M$ and define the Bianchi operator
\begin{equation}
\label{RDV}
X^{i}_{\bar{g}} (h) = (\bar{g} + h)^{ij} (\bar{g} + h)^{pq} \left( - \nabla^{\bar{g}}_{p} h_{qj} + \frac{1}{2} \nabla^{\bar{g}}_{j} h_{pq} \right),
\end{equation}
which assigns a vector field to every symmetric 2-form $h$ on $M.$
Let $(\Phi_{t})_{t \in I}$ be the flow generated by the time-dependent family of vector fields $X_{\bar{g}}(g_{t}),$ i.e., 
\begin{equation}
\label{RDD}
\frac{\partial}{\partial t} \Phi_{t} = X_{\bar{g}} (g_{t}) \circ \Phi_{t}~~\mathrm{with}~~\Phi_{0} = \mathrm{id}_{M}.
\end{equation}
We call this flow the \textit{Ricci-DeTurck deiffeomorphism} associated to $g_{t}$ below.
Then $\tilde{g}_{t} := \Phi^{*}_{t} g_{t}$ satisfies the \textit{Ricci flow equation}:
\begin{equation}\label{eq-ricciflow}
\frac{\partial}{\partial t} \tilde{g}_{t} = -2\, \mathrm{Ric}(\tilde{g}_{t})~~\mathrm{with}~~\tilde{g}_{0} = g_{0} = g.
\end{equation}
For each $g_{i}$, from Lemma \ref{lemm-basic}, we have a Ricci-DeTurck flow $g_{i}(t)$ defined on $[0, \tau]$ for some positive time $\tau$ (i.e., $g_{i}(t)$ satisfies (\ref{eq-RDF}) for all $t \in [0, \tau]$), which is independent of $i$.
We also have the corresponding Ricci flow $\tilde{g}_{i}(t)$ and the Ricci-DeTurck diffeomorphism $\Phi_{i, t}$ both defined on the same interval $[0, \tau]$.
Moreover, under the condition of Lemma \ref{lemm-1} $(e)$, we can see that for each $t_{0} \in (0,\tau),$
$\{ \tilde{g}_{i}(t) \}_{t \in [t_{0} , \tau]}$ smoothly subconverges to $\{ \tilde{g}(t) \}_{t \in [t_{0}, \tau ]}.$
\end{rema}

The next lemma will be used in the proof of Main Theorem \ref{theo-2}.
In particular, the assumption $p > n^{2}/2$ is needed for this lemma.
\begin{lemm}[Subconvergence of the Ricci-DeTurck diffeomorphisms]
\label{lemm-sub}
Let $M$ be a smooth manifold and $g$ a $C^{2}$-Riemannian metric on $M.$
Assume a sequence of $C^{2}$-Riemannian metrics $(g_{i})$ on $M$ converges to $g$ on $M$ in the $W^{1, p}~(p > n^{2}/2)$-sense.
let $(\Phi_{i,t})_{t \in [0, \tau)}$ be the corresponding Ricci-DeTurck diffeomorphism of $g_{i}$ with background metric $g$ defined as in (\ref{RDD}).
 Then, there is a subsequence of $(\Phi_{i,t})_{t \in [0, \tau/2]}$ that converges to a time-dependent map $(\Phi_{t})_{t \in [0, \tau/2]}$ such that
 \begin{itemize}
     \item $\Phi_{t}$ is a homeomorphism for all $t \in [0, \tau/2]$,
     \item for all $t \in [0, \tau/2]$, $\Phi_{t}, \Phi_{t}^{-1}$ are continuously differentiable and $\mathrm{d} \Phi_{t}$ and $\mathrm{d} \Phi^{-1}_{t}$ are mutually inverse,
     \item there is a subsequence $(\Phi_{i_{k}, t})_{t \in [0, \tau/2]}$ such that for all $t \in [0, \tau/2],~\Phi_{i_{k}, t}, \Phi_{i_{k}, t}^{-1}, \mathrm{d} \Phi_{i_{k}, t}$ and $\mathrm{d} \Phi_{i_{k}, t}^{-1}$ are converges to $\Phi_{t}, \Phi_{t}^{-1}, \mathrm{d} \Phi_{t}$ and $\mathrm{d} \Phi_{t}^{-1}$ respectively,
     \item $\Phi_{0} = \mathrm{id}_{M}$, where $\mathrm{id}_{M} : M \rightarrow M$ denotes the identity map.
 \end{itemize}
\end{lemm}

\begin{proof}
First note that there are a positive time $\tau = \tau(M, g) > 0$ and a positive constant $C = C(M, g) < \infty$ such that  for sufficiently large $i,$ Lemma \ref{lemm-basic} holds for $g$ and $g_{i}.$

\smallskip
\noindent
\underline{\textbf{Step 1($C^{0}$-convergence)}}:
From Lemma \ref{lemm-basic} $(d')$ and the definition (\ref{RDV}), the norm (with respect to $g$) of the time-dependent vector field defined in (\ref{RDV}) is bounded by some positive constant  $C = C(M, g).$ 
Thus, applying Gronwall's lemma to (\ref{RDD}), by the same argument in the proof of Lemma 2.1 in \cite{burkhardt2019pointwise}, one can obtain that
\begin{equation}
\label{ineq-1}
d_{g}(\Phi_{i,t}(p), \Phi_{i,s}(p)) \le C |t^{1 - \frac{n}{2p}} - s^{1 - \frac{n}{2p}}|,~~~~t,s \in [0, \tau/2],~p \in M.
\end{equation}
Here, $C$ depends only on $M, g$ and $\tau.$
Unlike that in {\cite[Lemma 2.1]{burkhardt2019pointwise}}, one can get the above type of estimate (in particular, the right-hand side is not $|\sqrt{t} - \sqrt{s}|$ but $|t^{1 - \frac{n}{2p}} - s^{1 - \frac{n}{2p}}|$).
This follows from the first-derivative estimate (Lemma \ref{lemm-basic} $(d')$) and hence $|X_{g}(g_{i, t})|_{g}$ is bounded by $C t^{\frac{n}{2p}}$ where $C = C(M, g)$ is a positive constant.
On the other hand, taking the derivative of both sides of (\ref{RDD}), and using the estimate of the second derivatives (Lemma \ref{lemm-basic} $(d')$), we obtain that
\[
 \frac{\partial}{\partial t}\, |\mathrm{d} \Phi_{i,t} |_{g} \le \frac{C}{t^{\frac{n}{4p} + \frac{3}{4}}} |\mathrm{d} \Phi_{i,t}|_{g},~~~~t \in (0, \tau/2],
\]
where $|\mathrm{d} \Phi_{i,t}|_{g}$ denotes the maximum of the operator norm of $\mathrm{d} \Phi_{i,t} : TM \rightarrow TM$ with respect to $g$ on $M.$ 
Hence, for all $t \in ( 0, \tau/2],$ we have
\[
\frac{d}{dt} \left( \frac{u(t)}{v(t)} \right) \le 0,
\]
where $u(t) = |\mathrm{d} \Phi_{i,t}|_{g}$ and $v(t) = e^{C t^{1 - \frac{n}{4p} -\frac{3}{4}}}$.
Then, since $u(t)$ is continuous on $[0, \tau/2],$ from the mean value theorem and this estimate of time-derivative, we have
\begin{equation}
\label{ineq-0}
|\mathrm{d} \Phi_{i,t}|_{g} \le C \exp{t^{1 - \frac{n}{4p} -\frac{3}{4}}},~~~~t \in [0, \tau/2].
\end{equation} 
Since the pullback metric $\Phi^{*}_{i,t}\, g_{i,t}$ satisfies the Ricci flow equation:
\[
\frac{\partial}{\partial t} (\Phi^{*}_{i,t}\, g_{i,t}) = -2\, \mathrm{Ric}(\Phi^{*}_{i,t}\, g_{i,t}),
\]
by Lemma \ref{lemm-basic} $(d')$ and the above estimate, there is a constant $C = C(M, g, \tau)$ such that for all points $(p,t) \in M \times (0, \tau/2]$ and all vectors $v \in T_{p} M,$
\[
\left| \frac{d}{dt} |v|^{2}_{\Phi^{*}_{i,t} g_{i,t}} \right| \le \frac{C}{t^{\frac{n}{4p} + \frac{3}{4}}} |v|^{2}_{\Phi^{*}_{i,t} g_{i,t}}.
\]
From this estimate, Lemma \ref{lemm-basic} $(b)$ and the mean value theorem, we obtain
\[
e^{-C \tau^{1 - \frac{n}{4p} - \frac{3}{4}}} |v|^{2}_{g_{i}} \le |v|^{2}_{\Phi^{*}_{i,t} g_{i,t}} \le e^{C \tau^{1 - \frac{n}{4p} - \frac{3}{4}}} |v|^{2}_{g_{i}}.
\]
Since $g_{i}$ converges to $g$ in the $C^{0}$-sense, for all sufficiently large $i,$ we have
\[
e^{-\tilde{C} \tau^{1 - \frac{n}{4p} - \frac{3}{4}}} |v|^{2}_{g} \le |v|^{2}_{\Phi^{*}_{i,t} g_{i,t}} \le e^{\tilde{C} \tau^{1 - \frac{n}{4p} - \frac{3}{4}}} |v|^{2}_{g}
\]
for some constant $\tilde{C} = \tilde{C}(M, g, \tau).$
Therefore, from Lemma \ref{lemm-basic} $(a),$ the $C^{0}$-closedness of $g_{i}$ and $g,$ and the previous estimate, there is a positive constant $C = C(M, g, \tau)$ such that for all sufficiently large $i,$ 
\begin{equation}
\label{ineq-2}
d_{g}(\Phi_{i,t}(p), \Phi_{i,t}(q)) \le C d_{g_{i,t}}(\Phi_{i,t}(p), \Phi_{i,t}(q)) \le C\, d_{g}(p, q).
\end{equation}
Combining the inequalities (\ref{ineq-1}) and (\ref{ineq-2}), we have
\[
d_{g}(\Phi_{i,t}(p), \Phi_{i,s}(q)) \le C \left( |t^{1 - \frac{n}{2p}} - s^{1 - \frac{n}{2p}}| + d_{g}(p,q) \right),~~~t,s \in [0, \tau/2],~p, q \in M.
\]
Then, by Arzel{\` a}-Ascoli theorem, a subsequence $(\Phi_{i_{k},t})_{t \in [0, \tau/2]}$ of $(\Phi_{i,t})_{t \in [0, \tau/2]}$ converges to a time-dependent map $\Phi_{t} : M \rightarrow M~(t \in [0, \tau/2])$ as $i_{k} \rightarrow \infty.$
In exactly the same way, one can prove that there is a subsequence $(\Phi^{-1}_{i_{k_{l}},t})_{t \in [0, \tau/2]}$ of $(\Phi^{-1}_{i_{k},t})_{t \in [0, \tau/2]}$ that converges to some time-dependent map $\tilde{\Phi}_{t} : M \rightarrow M~(t \in [0, \tau/2])$ as $i_{k_{l}} \rightarrow \infty.$ 
But, since $\Phi^{-1}_{i_{k_{l}},t} \circ \Phi_{i_{k_{l}} ,t} = \mathrm{id}_{M},$
$\Phi^{-1}_{t}$ exists and $\Phi^{-1}_{t} = \tilde{\Phi}_{t}$ for each $t \in [0, \tau/2].$
To prevent complicated in expression, we will simply write this converging subsequence $(\Phi_{i_{k_{l}}, t})_{t \in [0, \tau/2]}$ as $(\Phi_{i, t})_{t \in [0, \tau/2]}.$

\smallskip
\noindent
\underline{\textbf{Step 2($C^{1}$-convergence)}}:
Next, we will show that the first derivatives of $\Phi_{i,t}$ subconverges as $i \rightarrow \infty.$
Let $0 < a < 1.$ 
$\nabla_{g} g'_{t}$ satisfies a parabolic type PDE (see {\cite[Section~4,~Equation~(4)]{shi1989deforming}}); 
\[
\frac{\partial}{\partial t} \left( \nabla_{g} g'_{t} \cdot t^{a} \right) - \Delta_{g} \left( \nabla_{g} g'_{t} \cdot t^{a} \right) = F(g, \nabla_{g} g'_{t}, \nabla^{2}_{g} g'_{t}, \mathrm{Rm}_{g}, \nabla_{g} \mathrm{Rm}_{g}),
\]
and from the derivative estimate Lemma \ref{lemm-basic},
we have
\[
\left\| F(g, \nabla_{g} g'_{t}, \nabla^{2}_{g} g'_{t}, \mathrm{Rm}_{g}, \nabla_{g} \mathrm{Rm}_{g}) \right\|_{g} \le a C t^{a-1-\frac{n}{2p}} + C t^{a-\frac{n}{2p}} + C t^{a-\frac{n}{2p} - \frac{n}{4p} - \frac{3}{4}} + C t^{a-\frac{3n}{2p}}.
\]
for some positive constant $C$.
Thus, by the parabolic $W^{2,1}_{q}$-estimate ({\cite[Ch. 4, Section 3, Theorem 7]{krylov2008lectures}}), we have
\[
||\nabla_{g} g'_{t} \cdot t^{a} ||_{W^{2,1}_{q}(M \times (0, \tau/2], g)} \le C
\]
for all $q$ satisfying $\left( 1 + \frac{n}{2p} - a \right) q < 1$ and $q \ge 2.$
Note that if $\left( 1 + \frac{n}{2p} - a \right) q < 1$ then it holds that
\[
\left( \frac{3n + 3p}{4p} - a \right) q,~\left( \frac{n}{2p} - a \right) q,~\left( \frac{3n}{2p} - a \right) q < 1.
\]
(Recall our assumption $p > n^{2}/2 \ge n$.)
In particular, we have
\[
|| \nabla_{g} g'_{t} ||_{W^{2}_{q}(M, g)} \le C t^{-a}
\]
for all such $a, q$ and $t \in (0, \tau/2].$
Since, $p > n^{2}/2,$ we can choose $0 < a < 1$ sufficiently close to $1$
so that 
\begin{equation}\label{eq-region}
  n < \frac{1}{1 + \frac{n}{2p} - a}.
\end{equation}
Hence, as noted above, the weaker relations:
\[
n < \frac{1}{\frac{3n + 3p}{4p} - a},~\frac{1}{\frac{n}{2p} - a},~\frac{1}{\frac{3n}{2p} - a}
\]
are satisfied as well under the above condition (\ref{eq-region}).
Thus, from Morrey's embedding theorem,
\[
|| \nabla^{2}_{g} g'_{t} ||_{C^{\alpha}(M, g)} \le C t^{-a},~~~~t \in (0, \tau/2]
\]
for some $0 < \alpha < 1.$
Therefore, by the same arguments that derived (\ref{ineq-0}) above, we obtain that
\begin{equation}
\label{ineq-3}
|\mathrm{d} \Phi_{i,t}|_{C^{\alpha}(M, g)} \le C \exp{t^{1-a}},~~~~t \in [0, \tau/2].
\end{equation}
Then, by the inequalities (\ref{ineq-0}) and (\ref{ineq-3}), we can apply the Arzel{\` a}-Ascoli theorem and obtain that there is a subsequence $(\Phi_{i_{k} ,t})_{t \in [0, \tau/2]}$ such that as $i_{k} \rightarrow \infty,$
$(\Phi_{i_{k} ,t})_{t \in [0, \tau/2]} \rightarrow (\Phi_{t})_{t \in [0, \tau/2]}$ for some time-dependent map $(\Phi_{t})_{t \in [0, \tau/2]}.$ Moreover,
$(\Phi_{t})_{t \in [0, \tau/2]}$ are differentiable on $M$ and
$(\mathrm{d} \Phi_{i_{k} ,t})_{t \in [0, \tau/2]} \rightarrow (\mathrm{d} \Phi_{t})_{t \in [0, \tau/2]}$ as $i_{k} \rightarrow \infty.$
Similarly, there is a subsequence $(\Phi_{i_{k_{l}} ,t})_{t \in [0, \tau/2]}$ of $(\Phi_{i_{k} ,t})_{t \in [0, \tau/2]}$ such that
\[
(\Phi^{-1}_{i_{k_{l}} ,t})_{t \in [0, \tau/2]} \rightarrow (\tilde{\Phi}_{t})_{t \in [0, \tau/2]}~~\mathrm{and}~~
(\mathrm{d} \Phi^{-1}_{i_{k_{l}} ,t})_{t \in [0, \tau/2]} \rightarrow (\mathrm{d} \tilde{\Phi}_{t})_{t \in [0, \tau/2]}
\]
as $i_{k_{l}} \rightarrow \infty.$
Since $\Phi_{i_{k_{l}} ,t} \circ \Phi^{-1}_{i_{k_{l}} ,t} =\Phi^{-1}_{i_{k_{l}} ,t} \circ \Phi_{i_{k_{l}} ,t}  = \mathrm{id}_{M}$ and $\mathrm{d} \Phi_{i_{k_{l}} ,t} \circ \mathrm{d} \Phi^{-1}_{i_{k_{l}} ,t} = \mathrm{d} \Phi^{-1}_{i_{k_{l}} ,t} \circ \mathrm{d} \Phi_{i_{k_{l}} ,t} = \mathrm{id}_{TM},$ $\mathrm{d} \Phi_{t}$ are invertible and $\mathrm{d} \Phi_{t}^{-1} = \mathrm{d} \tilde{\Phi}_{t}$ for all $t \in [0, \tau/2].$
Finally, $\Phi_{0} = \mathrm{id}_{M}$ easily follows from the definition (\ref{RDD}) and the above construction of $(\Phi_{t})_{t \in [0, \tau/2]}.$
\end{proof}

\begin{rema}
In Lemma \ref{lemm-sub}, it is not known that the limit $(\Phi_{t})_{t \in [0, \tau/2]}$ is the Ricci-DeTurck diffeomorphism of the Ricci flow (\ref{eq-ricciflow}) starting at $g$ with background metric $g.$
Therefore, let $(g_{t})_{t \in [0. \tau/2]}$ be the Ricci-DeTurck flow (\ref{eq-RDF}) starting at $g$ with background metric $g,$ then we don't know whether or not $(\Phi_{t}^{*} g_{t})_{t \in [0, \tau/2]}$ is the solution of the Ricci flow equation starting at $g.$
\end{rema}

Next, we need the following stability of the heat flow with the Ricci flow background, which is essentially proved by Lee and Tam in {\cite[Theorem~3.1]{lee2022rigidity}}.

\begin{lemm}[{\cite[Theorem~3.1]{lee2022rigidity}}]
\label{lemm-harmonic}
Let $(M^{n}, g_{0})$ be a closed manifold of dimension $n \ge 2$
and $\{ f_{i} \}$ a family of functions on $M$ satisfying the assumptions $(1)$ and $(2)$ in Main Theorem \ref{theo-2}.
Suppose $g(t)~(t \in [0, \tau])$ be a solution of the Ricci flow equation (\ref{eq-ricciflow}) starting at $g_{0}$
such that 
\begin{equation}\label{eq-curv-derivative}
|\mathrm{Rm}(g(t))| \le a \cdot t^{-\left( \frac{n}{4p} + \frac{3}{4} \right)}
\end{equation}
for some $a > 0$ on $( 0, \tau].$
Then for all $i_{0} \in \mathbb{N},$ there are positive constants $\tau_{0}= \tau_{0}(n, \Lambda, i_{0}) > 0$ and $C_{0} = C_{0} (n,a, \Lambda, i_{0}) > 0$ such that the following holds.
For all $i \ge i_{0},$ there exists $F_{i}(t) \in C^{\infty}(M)~(t \in (0, \min \{ \tau, \tau_{0} \})$
satisfies the heat flow equation:
\begin{equation}\label{eq-heatflow}
\frac{\partial}{\partial t} F_{i}(t) = \Delta_{g_{t}} F_{i}(t)
\end{equation}
such that
\begin{itemize}
\item[(A)] $\left( \Lambda^{-2} - 2(n-1)t \right) (F_{i})^{*} g_{Eucl} \le g(t),$
\item[(B)] $\sup_{x \in M} d_{g_{Eucl}} (F_{i}(x,t), f_{i}(x)) \le C_{0} \sqrt{t}.$
Here, $g_{Eucl}$ denotes the Euclidean metric on $\mathbb{R}.$
\end{itemize}
Moreover, for any integer $l \ge 0,$ there is a constant $C = C(n, l, a, \Lambda, i_{0}) > 0$
such that for all $t \in ( 0, \min \{ \tau, \tau_{0} \} ],$
\begin{equation}\label{eq-heat-derivative}
|\nabla^{l} dF_{i}| \le C \cdot t^{-l/2 - \frac{n}{4p} + \frac{1}{4}}.
\end{equation}
\end{lemm}
\begin{proof}
Since $M$ is compact, the image of $f$ is compact in the target space $\mathbb{R}.$
From the assumption $(2)$ in Main Theorem \ref{theo-2},
there is a compact neighborhood $N (\subset \mathbb{R})$ of the image of $f$
such that the image of $f_{i}$ is contained in $N$ for all $i \ge i_{0}.$
Then, from the assumption $(1)$ of Main Theorem \ref{theo-2}, one can apply the proof of Lemma 3.1 and Theorem 3.1 in \cite{lee2022rigidity}
to $(M, g(t)),~(N, h := g_{Eucl}|_{N})$ and $f_{i} : M \rightarrow N \subset \mathbb{R}~(\mathrm{for~all}~i \ge i_{0}).$
Hence we obtain the desired assertions except for (\ref{eq-heat-derivative}) from Theorem 3.1 in \cite{lee2022rigidity}.
Therefore, we give a sketch of proof of (\ref{eq-heat-derivative}) in the rest of the proof.
First of all, from the assumption (\ref{eq-curv-derivative}) and the estimates of Shi, we have
\[
\left| \nabla^{k} \mathrm{Rm}(g(t)) \right| \le C t^{-\left( \frac{n}{4p} + \frac{3}{4} + \frac{k}{2} \right)}
\]
for any integer $k \ge 0$, where a positive constant $C$ depends only on $a, n$ and $k$. 
We will write the constants that depend only on $n, \min \{ \tau, \tau_{0} \}, \Lambda, a, l$ with the same symbol $C$.
And, we denote $\left| D^{l} dF_{i} \right|^{2}$ by $P_{l}$.
Then, as done in {\cite[Proof of Theorem 3.1]{lee2022rigidity}}, we have
\begin{equation}\label{eq-p-estimate}
\left( \frac{\partial}{\partial t} - \Delta_{g(t)} \right) P_{l} \le -2 P_{l} + C \left( t^{-\left( \frac{n}{4p} + \frac{3}{4} \right)} P_{l} + t^{-\left( \frac{n}{4p} + \frac{3}{4} + \frac{l}{2} \right)} P_{l}^{1/2} \right).
\end{equation}
Suppose the estimate holds true up to $l-1$.
Note that $P_{0}$ is the energy density and the desired estimate for $P_{0}$ is the result of {\cite[Theorem 3.1]{lee2022rigidity}}.
The following argument is similar to the one of {\cite[Proof of Theorem 3.1]{lee2022rigidity}}.

From (\ref{eq-p-estimate}), whenever $P_{l} > 0$, we have
\[
\begin{split}
\left( \frac{\partial}{\partial t} - \Delta_{g(t)} \right) \left( t^{A} P^{1/2}_{l} \right) &\le t^{A-1} P_{l}^{1/2} + C t^{A} \left( t^{-\left( \frac{n}{4p} + \frac{3}{4} \right)} P_{l}^{1/2} + t^{-\left( \frac{n}{4p} + \frac{3}{4} + \frac{l}{2} \right)} \right) \\
&= C \left( t^{\frac{l}{2} + \frac{n}{4p} - \frac{3}{4}} + t^{-1/2} \right),
\end{split}
\]
where $A = \frac{l}{2} + \frac{n}{4p} - \frac{1}{4} + \frac{1}{2}$.
On the other hand, from (\ref{eq-p-estimate}) and the induction assumption, we have
\[
\left( \frac{\partial}{\partial t} - \Delta_{g(t)} \right) P_{l-1} \le -2P_{l} + C \left( t^{-\left( \frac{n}{4p} + \frac{3}{4} \right)} t^{B} + t^{-\left( \frac{n}{4p} + \frac{3}{4} + \frac{l}{2} \right)} t^{B/2} \right),
\]
where $B = -\left( \frac{n}{4p} + \frac{3}{4} + \frac{l-1}{2} \right)$.
Hence, 
\[
\left( \frac{\partial}{\partial t} - \Delta_{g(t)} \right) \left( t^{D} P_{l-1} \right) \le -2t^{D} P_{l} + Ct^{D} \left( t^{-\left( \frac{n}{4p} + \frac{3}{4} \right)} t^{B} + t^{-\left( \frac{n}{4p} + \frac{3}{4} + \frac{l}{2} \right)} t^{B/2} \right) + t^{D-1} t^{B},
\]
where $D = l + \frac{n}{4p} + \frac{1}{2}$.
We notice that if $p > n$, 
\[
\left( \frac{\partial}{\partial t} - \Delta_{g(t)} \right) \left( t^{D} P_{l-1} \right) \le -2t^{D}P_{l} + Ct^{D} t^{-\left( \frac{3n}{8p} + \frac{5}{8} + \frac{3l-1}{4} \right)}
\]
and $D - \left( \frac{3n}{8p} + \frac{5}{8} + \frac{3l-1}{4} \right) \ge -\frac{1}{2}$.
Let $G := t^{A} P_{l}^{1/2} + t^{D} P_{l-1} - \alpha t^{1/2}$ for some $\alpha > 0$.
Then, 
\[
\begin{split}
\left( \frac{\partial}{\partial t} - \Delta_{g(t)} \right) G &\le C \left( t^{\frac{l}{2} + \frac{n}{4p} - \frac{3}{4}} + t^{-1/2} \right) -2t^{D} P_{l} + Ct^{-1/2} - \alpha t^{-1/2} \\
&\le C t^{-1/2} - \alpha t^{-1/2}.
\end{split}
\]
Therefore, if we take $\alpha > 0$ sufficiently large, we have 
\[
\left( \frac{\partial}{\partial t} - \Delta_{g(t)} \right) G < 0
\]
at the points where $P_{l} > 0$.
Since $G = 0$ at $t = 0$, we conclude that $G \le 0$, hence, 
\[
t^{A - 1/2} P_{l}^{1/2} \le 0.
\]
This completes the proof of (\ref{eq-heat-derivative}).
\end{proof}

The following lemma is a key to proving our main theorems.
Note that we only need $p > n$ for it.
\begin{lemm}[cf.~{\cite[Lemma~4]{bamler2016proof}}]
\label{lemm-key}
Let $M^{n}$ be a closed $n$-manifold ($n \ge 2$) and $g$ a $C^{2}$-Riemannian metric on $M.$
Suppose $f : M \rightarrow \mathbb{R}$ be a $\Lambda$-Lipschitz function for some $\Lambda > 0.$
Assume that a $C^{2}$-Riemannian metric $g'$ on $M$ is sufficiently close to $g$ in the $W^{1,p}(M, g)$-sense so that Lemma \ref{lemm-basic} holds for $g'$. 
Then for any given positive constant $\delta > 0,$ there is a constant
$\tau = \tau \left( M, g, |g - g'|_{W^{1, p}(M,g)}, \delta, \Lambda \right) > 0$ such that the following holds:
Assume that $\int_{M} R^{a,b,c}_{f}(g')\, e^{-f} d\mathrm{vol}_{g'} \ge \alpha$
for some $\alpha \in \mathbb{R}$ and $R(g') \ge 0$ on $M.$
Then there is a solution $(g'_{t})_{t \in [0, \tau]}$ to the Ricci flow equation (\ref{eq-ricciflow}) with the initial metric $g'$ and a solution $(f_{t})_{t \in (0, \tau]}$ of the heat flow equation (\ref{eq-heatflow}) with the Ricci flow background such that $\int_{M} R^{a,b,c}_{f_{t}}(g'_{t})\, e^{-f_{t}} d\mathrm{vol}_{g'(t)} \ge \alpha - \delta$
for all $t \in [0, \tau].$
\end{lemm}

\begin{proof}
From Lemmas \ref{lemm-basic} and \ref{lemm-harmonic}, there is a sufficiently small $1 > \tau' > 0$ such that there is a Ricci flow $(g'(t))_{t \in [0, \tau']}$ starting at $g'$ and
there is a heat flow $(f_{t})_{t \in (0, \tau']}$ with $f_{t} \rightarrow f$ as $t \searrow 0$ uniformly. 
Along these $(g'_{t})$ and $(f_{t}),$ for all $t \in (0, \tau'],$
\[
\begin{split}
&\frac{d}{d t} \left( \int_{M} R^{a,b,c}_{f_{t}}(g'_{t})\, e^{-f_{t}} d\mathrm{vol}_{g'_{t}} \right) \\
&\overset{(1)}{=} \int_{M} \left( \Delta_{g'_{t}} R(g'_{t}) + 2 |\mathrm{Ric}(g'_{t})|^{2}_{g'_{t}} - R(g'_{t})^{2} +2a\, g(\nabla \Delta f_{t}, \nabla f_{t}) -2a\, \mathrm{Ric}(\nabla f_{t}, \nabla f_{t}) \right)\, e^{-f_{t}} d\mathrm{vol}_{g'_{t}} \\
&~~~~~+b c \int_{M} (\Delta_{g'_{t}} f_{t}) e^{c-f_{t}}\, d\mathrm{vol}_{g'_{t}} + \int_{M} \left( \frac{\partial}{\partial t} e^{-f_{t}} \right) R_{f_{t}}(g'_{t})\, d\mathrm{vol}_{g'_{t}} \\
&\overset{(2)}{=} \int_{M} \left( 2|\mathrm{Ric}(g'_{t})|^{2}_{g'_{t}} -R(g'_{t})^{2} \right)\, e^{-f_{t}} d\mathrm{vol}_{g'_{t}} -2 \int_{M} \left( \Delta_{g'_{t}} f_{t} - |\nabla f_{t}|^{2} \right) R(g'_{t})\, e^{-f_{t}} d \mathrm{vol}_{g'_{t}} \\
&~~~~~+b c \int_{M} (\Delta_{g'_{t}} f_{t}) e^{c-f_{t}}\, d\mathrm{vol}_{g'_{t}} + a \int_{M} \left( \Delta |\nabla f_{t}|^{2} -2 |\nabla^{2} f_{t}|^{2} -4\mathrm
{Ric}(\nabla f_{t}, \nabla f_{t})\right)e^{-f_{t}}\, d\mathrm{vol}_{g'_{t}} \\
&\overset{(3)}{\ge} \int_{M} \left( \frac{2}{n}R(g'_{t})^{2} -R(g'_{t})^{2} \right)\, e^{-f_{t}} d\mathrm{vol}_{g'_{t}} -2 \int_{M} \left( \Delta_{g'_{t}} f_{t} - |\nabla f_{t}|^{2} \right) R(g'_{t})\, e^{-f_{t}} d \mathrm{vol}_{g'_{t}} \\
&~~~~~+ a \int_{M} \left( |\nabla f_{t}|^{2} \left( -\Delta f_{t} + |\nabla f_{t}|^{2} \right) -2|\nabla^{2} f_{t}|^{2} -4\mathrm{Ric}(\nabla f_{t}, \nabla f_{t}) \right)e^{-f_{t}}\, d\mathrm{vol}_{g'_{t}} \\
&~~~~~+b c \int_{M} (\Delta_{g'_{t}} f_{t}) e^{c-f_{t}}\, d\mathrm{vol}_{g'_{t}} \\
&\overset{(4)}{\ge} \left( \frac{2}{n} - 1 \right) \int_{M} R(g'_{t})^{2}\, e^{-f_{t}} d\mathrm{vol}_{g'_{t}} -2 \int_{M} \left( \Delta_{g'_{t}} f_{t} - |\nabla f_{t}|^{2} \right) R(g'_{t})\, e^{-f_{t}} d\mathrm{vol}_{g'_{t}} \\
&~~~~~-C_{0}t^{-\frac{n}{2p} + \frac{1}{2}} \int_{M} |\nabla f_{t}|^{2} e^{-f_{t}}\, d\mathrm{vol}_{g'_{t}}  -C_{1} t^{-\left( \frac{n}{2p} - \frac{1}{2} \right)} t^{-\frac{1}{2} -\frac{n}{4p} + \frac{1}{4}} \\
&~~~~~- C_{2} t^{-1 -\frac{n}{2p} + \frac{1}{2}} - C_{3} t^{-\left( \frac{n}{4p} + \frac{3}{4} \right)} t^{-\left( \frac{n}{2p} - \frac{1}{2} \right)} - C_{4} t^{-\frac{1}{2} -\frac{n}{4p} + \frac{1}{4}} \\
&\overset{(5)}{\ge} -C_{5} t^{-\frac{n}{4p} - \frac{3}{4}} \int_{M} R_{f_{t}}(g'_{t})\, e^{-f_{t}}d\mathrm{vol}_{g'_{t}} - C_{6} t^{-\frac{3n}{4p} - \frac{1}{4}}.
\end{split}
\]
Here, $C_{0}, \cdots, C_{6}$ are positive constants depending on $M, \Lambda$ and $|g -g'|_{W^{1, p}(M,g)}$.
And, each equality/inequality is derived as follows.
\begin{itemize}
    \item $(1)$ follows from the evolution of the scalar curvature and the volume form under the Ricci flow, and $\partial _{t} f_{t} = \Delta_{g_{t}} f_{t}~(t \in (0, \tau'])$.
    \item $(2)$ follows from applying the divergence formula to the term 
    \[
    \int_{M} \Delta_{g'(t)} R(g'(t))\, e^{-f_{t}} d\mathrm{vol}_{g'(t)},
    \]
    and Bochner formula to the term 
    \[
    g(\nabla \Delta f_{t}, \nabla f_{t}).
    \]
    \item $(3)$ follows from the Cauchy-Schwarz inequality $n |\mathrm{Ric}(g'_{t})|^{2} \ge R(g'_{t})^{2}$, and applying the divergence formula to the term
    \[
    \int_{M} \Delta |\nabla f_{t}|^{2} e^{-f_{t}}\, d\mathrm{vol}_{g'_{t}}.
    \]
    \item $(4)$ follows from the application of the derivative estimates (\ref{eq-heat-derivative}) in Lemma \ref{lemm-harmonic} to the third integral term.
    \item $(5)$ have been obtained as follows. From the assumption $R(g') \ge 0,$ by the maximum principle under the Ricci flow, we have $R(g'_{t}) \ge 0.$
Since the scalar curvature is invariant under the pullback action by a diffeomorphism, from the $C^{1}$-closedness assumption, one can apply the same derivative estimate $(d')$ in Lemma \ref{lemm-basic} to one $R(g'_{t})$ in the integrand of the left-hand side of the inequality.
Moreover, applying the derivative estimate (\ref{eq-heat-derivative}) in Lemma \ref{lemm-harmonic} to the second integrand of the left-hand side of the last inequality, we obtain the desired estimate.
\end{itemize}
Integrating the above inequality and using the continuity of $t \mapsto \int_{M} R_{f_{t}}(g'_{t}) e^{-f_{t}}\, d\mathrm{vol}_{g'_{t}}$ (especially at $t = 0$), we have
\[
\begin{split}
\int_{M} R_{f_{t}}(g'_{t}) e^{-f_{t}}\, d\mathrm{vol}_{g'_{t}} &\ge \int_{M} R_{f_{0}}(g'_{0}) e^{-f_{0}}\, d\mathrm{vol}_{g'_{0}} - C_{6} t^{\frac{3}{4} - \frac{3n}{4p}} \\
&-C_{5} \int^{t}_{0} s^{-\frac{n}{4p} - \frac{3}{4}} \left( \int_{M} R_{f_{s}}(g'_{s})e^{-f_{s}}\, d\mathrm{vol}_{g'_{s}} \right)\, ds.
\end{split}
\]
Hence, applying Lemma \ref{lemm-gronwall} (in Appendix \ref{section-appendix}) with 
\[
\begin{split}
    u(t) &:= \int_{M} R_{f_{t}}(g'_{t}) e^{-f_{t}}\, d\mathrm{vol}_{g'_{t}}, \\
    \alpha(t) &:= - C_{6} t^{\frac{3}{4} - \frac{3n}{4p}},~~\mathrm{and}\\
    \beta(t) &:= C_{5} t^{-\frac{n}{4p} -\frac{3}{4}},
\end{split}
\]
we get
\[
\begin{split}
\int_{M} R_{f_{t}}(g'_{t}) e^{-f_{t}}\, d\mathrm{vol}_{g'_{t}} &\ge \left( C_{5} t^{\frac{1}{4} - \frac{n}{4p}} \right)^{-1} \int_{M} R_{f_{0}}(g'_{0}) e^{-f_{0}}\, d\mathrm{vol}_{g'_{0}} \\
&~~~~~-C_{6}\left( C_{5} t^{\frac{1}{4} - \frac{n}{4p}} \right)^{-1} \int^{t}_{0} \exp \left( C_{5} s^{\frac{1}{4} - \frac{n}{4p}} \right)\, s^{\frac{3}{4} - \frac{3n}{4p}}\, ds.
\end{split}
\]
Therefore, if we take sufficiently small $0 < \tau << \tau',$ the desired assertion holds for such a constant $\tau > 0$. 
\end{proof}

\section{Proof of Main Theorems}
\label{section-proof}
In this section, we prove the main theorems.
First, we prove Main Theorem \ref{theo-2} because we will use almost the same argument in the proof of another main theorem \ref{theo-1}.
The idea of proof here is the same as that of Bamler \cite{bamler2016proof}, by contradiction.
To do this, we suppose that the weighted total scalar curvature of the limiting metric is less than or equal to $\kappa - \delta$ for some $\delta > 0$.
On the other hand, combining Lemma \ref{lemm-basic}, \ref{lemm-sub}, \ref{lemm-harmonic} and \ref{lemm-key}, we can deduce that the weighted total scalar curvature of the limiting metric is greater than or equal to $\kappa - \delta/2.$
This contradicts our supposition.
We prescribe these arguments with more precision below.
\begin{proof}[Proof of Main Theorem \ref{theo-2}]
We show the assertion by contradiction.
Suppose that 
\[
\int_{M} R^{a,b,c}_{f}(g)\, e^{-f} d\mathrm{vol}_{g} < \kappa.
\]
Then there is a positive constant $\delta > 0$ such that
\[
\int_{M} R^{a,b,c}_{f}(g)\, e^{-f} d\mathrm{vol}_{g} \le \kappa - \delta  < \kappa.
\]
On the other hand, since $\int_{M} R^{a,b,c}_{f_{i}}(g_{i})\, e^{-f_{i}}\, d\mathrm{vol}_{g_{i}} \ge \kappa$, from Lemma \ref{lemm-basic}, Lemma \ref{lemm-key} and Lemma \ref{lemm-harmonic}, there are a $\tau = \tau (\delta),$
a Ricci flow $(\tilde{g}_{i,t})_{t \in [0, \tau]}$ and a heat flow $(f_{i, t})_{t \in (0, \tau]}$ such that
\[
\int_{M} R^{a,b,c}_{f_{i,t}}(\tilde{g}_{i, t})\, e^{-f_{i, t}} d\mathrm{vol}_{\tilde{g}_{i,t}} \ge \kappa - \frac{1}{2} \delta
\]
and there also exists a heat flow $(f_{t})_{t \in (0, \tau]}$ with $f_{t} \rightarrow f$ as $t \searrow 0$ uniformly on $M.$ 
By Lemma \ref{lemm-sub}, as $i \rightarrow \infty,$ we have a solution of the Ricci--DeTurck flow equation $(g_{t})_{t \in [0,\tau]}$ (retaking a sufficiently small $\tau$ if necessary) starting at $g$ and a time-dependent $C^{1}$-diffeomorphism $(\Phi_{t})_{t \in [0, \tau]}$ with $\Phi_{0} = \mathrm{id}_{M}$
such that 
\[
\int_{M} R^{a,b,c}_{f_{t}}(\Phi_{t}^{*} g_{t})\, e^{-f_{t}} d\mathrm{vol}_{\Phi_{t}^{*} g_{t}} =  \int_{M} R^{a,b,c}_{f_{t} \circ \Phi^{-1}_{t}}(g_{t})\, e^{-f_{t} \circ \Phi^{-1}_{t}} d\mathrm{vol}_{g_{t}}
\ge \kappa - \frac{1}{2} \delta
\]
for all $t \in (0,\tau].$
On the other hand, from the derivative estimates (\ref{eq-heat-derivative}) ($l = 1, 2$ respectively) in Lemma \ref{lemm-harmonic}, we have $f_{t} \rightarrow f$ uniformly and $\nabla f_{t} \rightarrow \nabla f~\mathrm{a.e.}$ as $t \rightarrow 0$.
Furthermore, since $\Phi_{0} = \mathrm{id}_{M}$ (hence $\mathrm{d}\Phi_{0} = \mathrm{id}_{TM}$ as well) and $R_{g_{t}} \rightarrow R_{g}~(t \rightarrow 0)$ by Lemma \ref{lemm-basic} $(b)$, we have
\[
\begin{split}
\int_{M} R^{a,b,c}_{f}(g)\, e^{-f} d\mathrm{vol}_{g} &= \int_{M} (R_{g} + a|\nabla f|^{2} +b e^{cf}) e^{-f}\, d\mathrm{vol}_{g} \\
&= \lim_{t \rightarrow 0} \int_{M} R^{a,b,c}_{f_{t} \circ \Phi^{-1}_{t}}(g_{t})\, e^{-f_{t} \circ \Phi^{-1}_{t}} d\mathrm{vol}_{g_{t}} \\
&\ge \kappa - \frac{1}{2} \delta > \kappa - \delta.
\end{split}
\]
This contradicts our supposition
\[
\int_{M} R^{a,b,c}_{f}(g)\, e^{-f} d\mathrm{vol}_{g} \le \kappa -\delta
\]
and concludes the proof.
\end{proof}

Next, we give a proof of Main Theorem \ref{theo-1}.
\begin{proof}[Proof of Main Theorem \ref{theo-1}]
The proof is similar to the one of Main Theorem \ref{theo-2}.
However, since each metric is Ricci soliton, the Ricci flow starting from such a metric is homothetic.
Using this fact, we can prove a similar statement to Theorem \ref{theo-2} under the assumption of lower $C^{0}$ regularity.
Indeed, the solution $(\tilde{g}_{i,t})_{t \in [0, \tau)}$ of the Ricci flow equation starting from $g_{i}$ constructed in Lemma \ref{lemm-1} is
\[
\tilde{g}_{i, t} = (1 -2\lambda_{i}t) \Phi^{*}_{i, t} (g_{i})~~~~t \in [0, \tau),
\]
where
$\Phi_{i, t}$ is the family of diffeomorphisms generated by the time-dependent vector field $X_{i}(t) := (1-2\lambda_{i} t)^{-1} Y_{i}$ with the initial condition $\Phi_{i, 0} = \mathrm{id}_{M},$ and $\tau > 0$ is the uniform existence time of the flows guaranteed in Lemma \ref{lemm-1}.
Thus, we have
\begin{equation}\label{eq-soliton}
  \int_{M} R(\tilde{g}_{i, t})\, d\mathrm{vol}_{\tilde{g}_{i, t}} = (1 -2 \lambda_{i} t)^{n/2 - 1} \int_{M} R(g_{i})\, d\mathrm{vol}_{g_{i}}.
\end{equation}
for all $t \in [0, \tau).$
We will first consider the case of $\kappa \ge 0.$
From (\ref{eq-soliton}), for any $\delta > 0,$ there is a positive time $\tau = \tau(\delta, C_{+}) > 0$ such that 
for any diffeomorphism $\Phi: M \rightarrow M$ and $t \in [0, \tau],$
\[
\int_{M} R(\Phi^{*} \tilde{g}_{i, t})\, d\mathrm{vol}_{\Phi^{*} g_{i, t}} = \int_{M} R(\tilde{g}_{i, t})\, d\mathrm{vol}_{\tilde{g}_{i, t}}
\ge \kappa - \delta.
\]
Similarly, when $\kappa < 0,$ we see that for any $\delta > 0,$ there is a positive time $\tau = \tau(\delta, C_{-}) > 0$ such that 
for any diffeomorphism $\Phi: M \rightarrow M$ and $t \in [0, \tau],$
\[
\int_{M} R(\Phi^{*} \tilde{g}_{i, t})\, d\mathrm{vol}_{\Phi^{*} \tilde{g}_{i, t}} = \int_{M} R(\tilde{g}_{i, t})\, d\mathrm{vol}_{\tilde{g}_{i, t}}
\ge \kappa - \delta.
\]
In particular, we take the inverse of the Ricci-DeTurck diffeomorphism (cf. Remark \ref{rema-lemm-1}) of $\tilde{g}_{i, t}$ as $\Phi$ here, then we have
\[
\int_{M} R(g_{i, t})\, d\mathrm{vol}_{g_{i, t}} \ge \kappa - \delta~~~\mathrm{for~all}~t \in [0, \tau],
\]
where $g_{i, t}$ is the Ricci-DeTurck flow starting at $g_{i}.$
Hence, we can use this claim instead of Lemma \ref{lemm-key} in the proof of Main Theorem \ref{theo-2}.
Therefore, using Lemma \ref{lemm-1} instead of Lemma \ref{lemm-basic}, we can prove the assertion in the same way as in the proof of Main Theorem \ref{theo-2}.
\end{proof}

\begin{rema}
If $g_{i}$ is a shrinking Ricci soliton (i.e., $\lambda_{i} > 0$ in the above situation),
then the maximal existence time of the corresponding Ricci flow is $(2 \lambda_{i})^{-1}.$
But, since we assume $g_{i}$ converges to $g$ in the $C^{0}$-sense, Lemma \ref{lemm-1} implicitly prevents $\lambda_{i} \rightarrow \infty$ as $i \rightarrow \infty.$
\end{rema}

\section{Some corollaries}
\label{section-coro}
We present a corollary of Fact \ref{theo-3} here.
In order to do this, we need to recall the definition of the Yamabe constant:
\begin{defi}[Yamabe constant]
The \textit{Yamabe constant} $Y(M, g)$ of a closed Riemannian manifold $(M, g)$ is defined as
\[
Y(M, g) := \inf \left\{ \int_{M} R(\tilde{g})\, d\mathrm{vol}_{\tilde{g}}~\middle|~\tilde{g} \in [g]~\mathrm{and}~\mathrm{Vol}(M,\tilde{g}) = 1 \right\},
\]
where $[g] := \left\{ \tilde{g} = u^{\frac{4}{n-2}} g~|~u \in C^{\infty}(M),~u > 0~\mathrm{on}~M \right\}$ is the conformal class of the metric $g.$
By the definition, $Y(M, g)$ depends only on the conformal class $[g]$ of $g.$
A Riemannian metric $\tilde{g} \in [g]$ with $\mathrm{Vol}(M, \tilde{g}) = 1$ is called \textit{Yamabe metric} of $[g]$ if 
\[
Y(M, g) = \int_{M} R(\tilde{g})\, d\mathrm{vol}_{\tilde{g}}.
\]
\end{defi}
\begin{coro}[of Fact \ref{theo-3}]
Let $M^{n}$ be a closed $n$-manifold and $g$ a $C^{2}$-Riemannian metric on $M.$
Let $(g_{i})$ be a sequence of $C^{2}$-Riemannian metrics on $M$ that 
converges to $g$ on $M$ in the $C^{0} \cap W^{1, 2}$-sense, and $\mathrm{Vol}(M, g_{i}) = 1.$
Assume that $g$ is a Yamabe metric of $[g]$ and there are a constant $\kappa \in \mathbb{R}$ such that 
$Y(M, g_{i}) \ge \kappa$.
Then $Y(M, g) \ge \kappa.$
\end{coro}
\begin{proof}
From the definition of $Y(M, g_{i})$ and $\mathrm{Vol}(M, g_{i}) = 1,$ we have 
\[
\int_{M} R(g_{i})\, d\mathrm{vol}_{g_{i}} \ge \kappa.
\]
Since $g_{i} \rightarrow g$ in the $C^{0}$-sense and $\mathrm{Vol}(M, g_{i}) = 1,$
we also have $\mathrm{Vol}(M, g) = 1.$
Hence, from Fact \ref{theo-3}, we have
\[
\int_{M} R(g)\, d\mathrm{vol}_{g} \ge \kappa.
\]
Therefore, since $g$ is a Yamabe metric of $[g]$ and $\mathrm{Vol}(M, g) = 1,$ we obtain
\[
Y(M, g) = \int_{M} R(g)\, d\mathrm{vol}_{g} \ge \kappa.
\]
\end{proof}
Up to this point, we have only dealt with case where the underlying manifold is closed.
In a special situation, we can prove a similar statement as Fact \ref{theo-3} for a noncompact manifold.
\begin{prop}[Conformal deformations on an open manifold]
\label{theo-conformal}
Let $(M^{n}, g)$ be a non-compact Riemannian $n$-manifold ($n \ge 3$) with $\int_{M} R(g)\, d\mathrm{vol}_{g} < +\infty$ and $u_{i} : M \rightarrow \mathbb{R}$ a sequence of positive $C^{2}$-functions.
Assume that each $u_{i}$ is equal to 1 outside a compact set and
\[
u_{i} \rightarrow 
1~\mathrm{uniformly}~C^{1}~\mathrm{sense~on}~M.
\]
Set 
$g_{i} :=
u_{i}^{\frac{4}{n-2}} g.$
Then, it holds that
\[
\int_{M} R(g_{i})\, d\mathrm{vol}_{g_{i}} \overset{i \rightarrow \infty}{\longrightarrow}  \int_{M} R(g)\, d\mathrm{vol}_{g}.
\]
\end{prop}

\begin{proof}
From the formula for the scalar curvature and the volume form under this conformal change:
\[
R(g_{i}) = -4 \frac{n-1}{n-2} u_{i}^{- \frac{n+2}{n-2}} \Delta_{g} u_{i} + u_{i}^{-\frac{4}{n-2}} R(g),~~~~\mathrm{vol}_{g_{i}} = u_{i}^{\frac{2n}{n-2}} \mathrm{vol}_{g},
\]
we have
\[
\begin{split}
\int_{M} R(g_{i})\, d\mathrm{vol}_{g_{i}}
&= -4 \frac{n-1}{n-2} \int_{M} u_{i}^{- \frac{n+2}{n-2}} \Delta_{g} u_{i} \left( u_{i}^{\frac{2n}{n-2}}\, d\mathrm{vol}_{g} \right)
+ \int_{M} u_{i}^{\frac{2n-4}{n-2}} R(g)\, d\mathrm{vol}_{g} \\
&= -4 \frac{n-1}{n-2} \int_{M} u_{i} \Delta_{g} u_{i}\, d\mathrm{vol}_{g} + \int_{M} u_{i}^{\frac{2n-4}{n-2}} R(g)\, d\mathrm{vol}_{g} \\
&= 4 \frac{n-1}{n-2} \int_{M} |\nabla_{g} u_{i}|_{g}^{2} d\mathrm{vol}_{g} + \int_{M} u_{i}^{\frac{2n-4}{n-2}} R(g)\, d\mathrm{vol}_{g}.
\end{split}
\]
We have used the divergence formula and the fact that $u_{i}$ is equal to 1 outside a compact set in the third equality.
Since $u_{i} \overset{i \rightarrow \infty}{\longrightarrow} 1$ in the $C^{1}$-uniformly sense on $M,$ we have
\[
\int_{M} |\nabla_{g} u_{i}|_{g}^{2} d\mathrm{vol}_{g} \rightarrow 0~~\mathrm{as}~i \rightarrow \infty
\]
and
\[
\int_{M} u_{i}^{\frac{2n-4}{n-2}} R(g)\, d\mathrm{vol}_{g} \rightarrow \int_{M} R(g)\, d\mathrm{vol}_{g}~~\mathrm{as}~i \rightarrow \infty
\]
Therefore we obtain the desired assertion from the above equality.
\end{proof}

Lee and LeFloch \cite{lee2015positive} defined a notion of distributional scalar curvature for smooth manifolds that have a metric tensor that has only certain lower regularity.
\begin{defi}[Distributional scalar curvature ({\cite[Definition 2.1]{lee2015positive}}, {\cite[Section 2]{jiang2023weak}})]
\label{defi-distr}
Let $M$ be a smooth manifold endowed with a smooth background metric $h.$
Given any Riemannian metric $g$ with $L^{\infty}_{loc}(M) \cap W^{1, 2}_{loc}(M)$ regularity and locally bounded inverse $g^{-1} \in L^{\infty}_{loc}(M),$ the \textit{scalar curvature distribution} $R_{g}$ is defined for every compactly supported smooth test function $u : M \rightarrow \mathbb{R}$ by
\begin{equation}\label{eq-distr-def}
\langle R_{g}, u \rangle := \int_{M} \left( -V \cdot \overline{\nabla}\left( u \frac{d\mathrm{vol}_{g}}{d\mathrm{vol}_{h}} \right) + F u \frac{d\mathrm{vol}_{g}}{d\mathrm{vol}_{h}} \right)\, d\mathrm{vol}_{h},
\end{equation}
where $V = (V^{k}) \in \Gamma(M)$ is given by $V^{k} := g^{ij} \Gamma^{k}_{ij} - g^{ik}\Gamma^{j}_{ji},$ $F$ is a function\footnote{The expression of $F$ in \cite{lee2015positive} was a little incorrect.} as 
\[
F := \mathrm{tr}_{g}\mathrm{Ric}_{h} - \overline{\nabla}_{k} g^{ij}\Gamma^{k}_{ij} + \overline{\nabla}_{k} g^{ik} \Gamma^{j}_{ji} + g^{ij} \left( \Gamma^{k}_{kl} \Gamma^{l}_{ij} - \Gamma^{k}_{jl} \Gamma^{l}_{ik} \right)
\]
and $\Gamma^{k}_{ij} := \frac{1}{2} g^{kl} \left( \overline{\nabla}_{i} g_{jl} + \overline{
\nabla}_{j} g_{il} - \overline{\nabla}_{l} g_{ij} \right).$
Here, $\overline{\nabla}$ denotes the Levi-Civita connection of $h.$

Let $\kappa$ be a continuous function on $M.$ We say that $R_{g} \ge \kappa$ \textit{in the distributional sense} if $\langle R_{g}, u \rangle - \int_{M} \kappa u\, d\mathrm{vol}_{g} \ge 0$ for any nonnegative compactly supported test function $u \in C^{\infty}_{+}(M) \cap C^{\infty}_{0}(M).$
\end{defi}
\begin{rema}
  \label{rema-distr-smooth}
  If a metric $g$ is $C^{2},$ then the scalar curvature distribution $\langle R_{g}, u \rangle$ coincides with $\int_{M} R(g) u\, d\mathrm{vol}_{g}.$
\end{rema}
For more details about the distributional scalar curvature and related results, see \cite{jiang2023weak, lee2015positive, sormani2023extreme, tian2023compactness}.
From Gromov's $C^{0}$-limit theorem (Theorem \ref{gromov-limit}), there has already been a definition of scalar curvature lower bounds for $C^{0}$ metrics (see {\cite[Definition 1.2]{lee2022rigidity}} for example).
Namely, a $C^{0}$ metric $g$ on a smooth manifold $M$ is of $R(g) \ge \kappa$ on $M$ in the Gromov's sense if and only if there exists a sequence of $C^{2}$ metrics $(g_{i})$ such that $g_{i}$ converge $C^{0}$-locally to $g$ and satisfy $R(g_{i}) \ge \kappa$ on $M.$
Note that Burkhardt-Guim \cite{burkhardt2019pointwise} pointed out that her definition (via the Ricci--DeTurck flow) and this Gromov's definition are actually equivalent on a closed manifold. 
For example, on tori, there is no metric $g$ which is of $R(g) \ge \kappa > 0$ in the Gromov's sense (or equivalently in the sense of {\cite[Definition 1.2]{burkhardt2019pointwise}}) from the resolution of Geroch's conjecture \cite{gromov1980spin, schoen1979existence, schoen1979structure}.
In contrast, a metric $g$ that is of $R(g) \ge \kappa > 0$ in the sense of the definition \ref{defi-distr} might exist on a torus.
At least on a manifold whose Yamabe invariant is nonpositive, the question of how different these definitions are is related to Schoen's conjecture (cf. \cite{cecchini2024positive, jiang2023weak, lee2025continuous}).
As a corollary of Corollary \ref{coro-weighted-scalar}, we can obtain the following. This is the same as Corollary \ref{coro-distr-0} in Section \ref{sec1}.
\begin{coro}[$=$ Corollary \ref{coro-distr-0}]
\label{coro-distr}
  Let $p > n^{2}/2$ and $\kappa$ a constant.
  Suppose that $M$ is a closed manifold of dimension $n \ge 2,$ $g$ is a $C^{2}$ Riemannian metric on $M$ and $(g_{i})$ is a sequence of $C^{2}$ metric on $M.$
  Assume the following:
  \begin{itemize}
    \item[(1)] a sequence $(\phi_{i})$ of nonnegative smooth functions on $M$ satisfying:
    for any positive constant $a > 0$ there is a positive constant $\Lambda > 0$ such that $\log (\phi_{i} + a)$ is $\Lambda$-Lipschitz on $M$ for all $i,$
    \item[(2)] $(\phi_{i})$ converges to some nonnegative continuous function $\phi$ in the uniformly $C^{0}$-sense on $M,$
    \item[(3)] $R(g_{i}) \ge 0$ on $M$ for each $i,$
    \item[(4)] $\int_{M} R(g_{i}) \phi_{i}\, d\mathrm{vol}_{g_{i}} \ge \kappa \int_{M} \phi_{i}\, d\mathrm{vol}_{g_{i}},$
    \item[(5)] $g_{i}$ converges to $g$ in the $W^{1, p}$-sense.
  \end{itemize}
  Then
  \[
  \int_{M} R(g) \phi\, d\mathrm{vol}_{g} \ge \kappa \int_{M} \phi\, d\mathrm{vol}_{g}.
             \]
\end{coro}
\begin{proof}
From $(3)$ and $(4),$ for any (small) positive constant $a > 0,$ 
\[
\begin{split}
  \int_{M} R(g_{i}) (\phi_{i}+ a)\, d\mathrm{vol}_{g_{i}} &\ge \kappa \int_{M} \phi_{i}\, d\mathrm{vol}_{g_{i}} \\
  &= \kappa \int_{M} \phi\, d\mathrm{vol}_{g} + \kappa \left( \int_{M} \phi_{i}\, d\mathrm{vol}_{g_{i}} - \int_{M} \phi\, d\mathrm{vol}_{g} \right).
\end{split}
\]
Then, applying Main Theorem \ref{theo-2} to $f_{i} = -\log (\phi_{i} + a)$ and $f = -\log (\phi + a),$ we obtain that
\[
  \int_{M} R(g) (\phi+ a)\, d\mathrm{vol}_{g_{i}} \ge \kappa \int_{M} \phi\, d\mathrm{vol}_{g} + \kappa \delta
  \]
for all $0 < \delta << 1.$
Here, $a, \delta > 0$ can be taken arbitrarily small, and we obtain the desired inequality:
\[
\int_{M} R(g) \phi\, d\mathrm{vol}_{g_{i}} \ge \kappa \int_{M} \phi\, d\mathrm{vol}_{g}.
\]
\end{proof}
\begin{ques}
Can we replace the condition $(3)$ and $(4)$ with ``$R(g_{i}) \ge \kappa$ in the distributional sense for some nonnegative constant $\kappa \ge 0$''?
\end{ques}

\noindent
If we make the regularity of convergence much stronger $W^{1, p}~(p > n^{2}/2)$ in {\cite[Theorem 3.2 (1)]{jiang2023weak}}, then it can be proven that the above question is positively true.

\section{Examples}
\label{section-examples}
In this section, we construct some examples of Riemannian metrics $(g_{i})$ on a smooth manifold $M$ such that 
$g_{i}$ converges to a metric $g$ in some sense and
\[
\int_{M} R(g_{i})\, d\mathrm{vol}_{g_{i}} \ge \kappa ~\mathrm{for~some~constant}~\kappa \in \mathbb{R},~\mathrm{but}~
\int_{M} R(g)\, d\mathrm{vol}_{g} < \kappa.
\]
\begin{itemize}
    \item In Example \ref{exam-integral}, $M = \mathbb{R}^{n} \setminus \{ o \}~(n \ge 3)$ and $g_{i} \rightarrow g_{Eucl}$ in the locally uniformly smooth sense in $\mathbb{R}^{n} \setminus \{ o \}$, but not in the $C^{0}$-sense. Note that the limiting metric $g_{Eucl}$ in $\mathbb{R}^{n} \setminus \{ o \}$ is incomplete.
    \item In Example \ref{exam-C10}, $M = \mathbb{R}^{n}~(n \ge 3)$ and $g_{i} \rightarrow g_{Eucl}$ in the uniformly $C^{0}$-sense, but not in the $W^{1,2}$-sense.
    \item In Example \ref{exam-closed}, $(M, g)$ is an arbitrary closed Riemannian $n$-manifold with $n \ge 3$ and $g_{i} \rightarrow g$ in the $C^{0}$-sense, but not in the $W^{1,2}$-sense.
    \item In Example \ref{exam-C21}, $M = \mathbb{R}^{n}~(n \ge 2)$ and $g_{i} \rightarrow g_{Eucl}$ in the uniformly $C^{1}$-sense, but not in the $C^{2}$-sense. 
\end{itemize}

If $M$ is a compact smooth manifold and $C^{2}$-Riemannian metrics $\{ g_{i} \}$ converges to a $C^{2}$-Riemannian metric $g$ on $M$ in the $C^{2}$ sense as $i \rightarrow \infty$,
then $\max_{M} R(g_{i}) \rightarrow \max_{M} R(g).$ 
So, by Lebesgue's dominated convergence theorem, we have
\[
\int_{M} R(g_{i})\, d\mathrm{vol}_{g_{i}} \rightarrow \int_{M} R(g)\, d\mathrm{vol}_{g}~~~~\mathrm{as}~~i \rightarrow \infty.
\]
However, if $M$ is non-compact, it is not known whether or not
there is a Lebesgue integrable function $f : M \rightarrow \mathbb{R}$ such that $|R(g_{i})| \le f$ a.e. on $M$ for all $i.$
Hence, in this situation, Lebesgue's dominated convergence theorem cannot be applied in general.
Indeed, the following example implies that Fact \ref{theo-3} does not hold in general if $M$ is non-compact.

\begin{exam}[$C^{\infty}$ locally uniformly convergence and incomplete limiting metric]
\label{exam-integral}
The limiting metric of the first example constructed below is incomplete on $\mathbb{R}^{n}~(n \ge 3)$.
Consider the smooth positive function $u_{i} : \mathbb{R}^{n} \rightarrow \mathbb{R}$ defined as
\[
u_{i} = 
\phi \left( i^{-1 + l} e^{-i r^{2}} \right)+ 1
\]
and
$\left( \mathbb{R}^{n},~g_{i} := u_{i}^{\frac{4}{n-2}} \cdot g_{Eucl} \right)~(n \ge 3).$
Here $\phi : \mathbb{R}^{n} \rightarrow [0,1]$ is a smooth cut-off function
such that $\phi \equiv 1$ on the closed ball $\overline{B_{r_{0}}} := \{ x \in \mathbb{R}^{n}|~r(x) \le r_{0} \}$ and
$\phi \equiv 0$ outside of the $\varepsilon/2$-neighbourhood $\left( \overline{B_{r_{0}}} \right)_{\varepsilon/2}$ of $\overline{B_{r_{0}}}$
where $r_{0} > 0$ is an arbitrarily fixed positive constant.
Here, $g_{Eucl}$ denotes the Euclidean metric on $\mathbb{R}^{n},~r : \mathbb{R}^{n} \rightarrow \mathbb{R}_{\ge 0}$ is the Euclidean distance function from the origin $o \in \mathbb{R}^{n}$, and
\[
l := \frac{n+2}{4}~\mathrm{when}
\begin{cases}
n = 2m + 1~(m \ge 1), \\
n = 2m~(m \ge 2).
\end{cases}
\]
Then for each $i,$ $(\mathbb{R}^{n}, g_{i})$ is a non-compact smooth Riemannian manifold with
\begin{equation}\label{eq-scal-conformal}
  \begin{split}
R(g_{i}) &= u_{i}^{- \frac{n + 2}{n - 2}} \left( - 4 \frac{n-1}{n-2} \Delta_{g_{Eucl}} u_{i} + R(g_{Eucl}) u_{i} \right) \\
&= u_{i}^{- \frac{n + 2}{n - 2}} \left( - 4 \frac{n-1}{n-2} \Delta_{g_{Eucl}} u_{i} \right).
\end{split}
\end{equation}
Moreover, $g_{i}$ converges to $g_{Ecul}$ in the locally $C^{\infty}$-sense in $\mathbb{R}^{n} \setminus \{ o \}$ but not in the $C^{0}$-sense on $\mathbb{R}^{n}.$
Note that $(\mathbb{R}^{n} \setminus \{ o \}, g_{Eucl})$ is incomplete.
On $\overline{B_{r_{0}}},$
\[
\begin{split}
|\nabla u_{i}|^{2} = \sum^{n}_{j=1} \left| \frac{\partial}{\partial x^{j}} i^{-1 + l} e^{-i r^{2}} \right|^{2} 
&= \sum^{n}_{j=1} \left| \frac{\partial r}{\partial x^{j}} \frac{\partial}{\partial r} i^{-1+l} e^{-i r^{2}} \right|^{2} \\
&= \sum^{n}_{j=1} \left| \frac{x^{j}}{r} (-2) i^{l} r e^{-i r^{2}} \right|^{2} \\
&= 4 i^{2l} r^{2} e^{-2 i r^{2}}.    
\end{split}
\]
When $i \rightarrow \infty$,
\[
R(g_{i}) \rightarrow 
\begin{cases}
0 & \mathrm{on}~\overline{B_{r_{0}}} \setminus \{ o \}, \\
\infty & \mathrm{at}~o.
\end{cases}
\]
Indeed, we can observe such a behavior from the form of the scalar curvature on $\overline{B_{r_{0}}}$ as follows.
\begin{equation}
\label{eq-scalar-curvature}
R(g_{i}) = 
-\frac{4(n-1)}{n-2} u_{i}^{-\frac{n+2}{n-2}} \left( -2 i^{l} n + 4 i^{l+1} r^{2} \right) e^{-i r^{2}}
\end{equation}
From (\ref{eq-scal-conformal}) and the divergence formula, we have
\begin{equation}
\label{eq-cal}
\begin{split}
\int_{\mathbb{R}^{n}} R(g_{i})\, d\mathrm{vol}_{g_{i}}
&= -4 \frac{n-1}{n-2} \int_{\mathbb{R}^{n}} u_{i}^{- \frac{n+2}{n-2}} \Delta_{g_{Eucl}} u_{i} \left( u_{i}^{\frac{2n}{n-2}}\, d\mathrm{vol}_{g_{Eucl}} \right) \\
&= -4 \frac{n-1}{n-2} \int_{\mathbb{R}^{n}} u_{i} \Delta_{g_{Eucl}} u_{i}\, d\mathrm{vol}_{g_{Eucl}} \\
&= 4 \frac{n-1}{n-2} \int_{\left( \overline{B_{r_{0}}} \right)_{\varepsilon}} |\nabla u_{i}|^{2} d\mathrm{vol}_{g_{Eucl}} \\
&\ge 16 \frac{n-1}{n-2} \int_{\overline{B_{r_{0}}}} i^{2l} r^{2} e^{-2 i r^{2}} d\mathrm{vol}_{g_{Eucl}} \\
&= 16 \frac{n-1}{n-2} \mathrm{Vol}(S^{n-1}) \int^{r_{0}}_{0} i^{2l} r^{2} e^{-2i r^{2}} r^{n-1}\, dr,
\end{split}
\end{equation}
where $\mathrm{Vol}(S^{n-1})$ denotes the volume of $(n-1)$-sphere with respect to the standard metric.
Here, 
\[
\begin{split}
\int^{r_{0}}_{0} r^{n+1} e^{-2i r^{2}} dr
&= \left[ -\frac{1}{2 i} r^{n} e^{-2 i r^{2}} \right]^{r_{0}}_{0}
+ \frac{n}{2i} \int^{r_{0}}_{0} r^{n-1} e^{-2i r^{2}} dr \\
&= -\frac{1}{2i} r_{0}^{n} e^{-2i r_{0}^{2}} + \frac{n}{2i} \int^{r_{0}}_{0} r^{n-1} e^{-2i r^{2}} dr
\end{split}
\]
Set the left hand side of this equation as
$I_{n+1} := \int^{r_{0}}_{0} r^{n+1} e^{ -2i r^{2}} dr.$
Then we have
\[
\begin{cases}
  I_{n+1} = 
-\frac{e^{-2i r_{0}^{2}}}{2i} \sum^{m}_{k = 0} \prod^{k}_{s = 0} \left( \frac{n - 2s}{2i} \right) + \prod^{m}_{s=0} \left( \frac{2s+1}{2i} \right)\, I_{0}  
&\mathrm{if}~n = 2m + 1~(m \ge 1), \\
I_{n+1} = 
-\frac{e^{-2i r_{0}^{2}}}{2i} \sum^{m - 1}_{k = 0} \prod^{k}_{s = 0} \left( \frac{n - 2s}{2i} \right) + \prod^{m}_{s=1} \left( \frac{2s}{2i} \right)\,  I_{1} 
&\mathrm{if}~n = 2m~(m \ge 2).
\end{cases}
\]
Moreover,
\[
\begin{split}
I_{0} = \int^{r_{0}}_{0} e^{-2i r^{2}}\, dr
&\ge \left( \int^{r_{0}}_{0} \int^{\frac{\pi}{2}}_{0} e^{-2i r^{2}} r dr d\theta \right)^{1/2} \\
&= \sqrt{\frac{\pi}{2} \left( \frac{1}{2i} - \frac{e^{-2ir_{0}^{2}}}{2i} \right)},
\end{split}
\]
and
\[
I_{1} = \int^{r_{0}}_{0} r e^{-2i r^{2}}\, dr
= \left[ -\frac{1}{2i} e^{-2ir^{2}} \right]^{r_{0}}_{0} = \frac{1}{2i} - \frac{1}{2i} e^{-2ir^{2}_{0}}.
\]
Combining these, as $i \rightarrow \infty,$ we can see that the rightmost term of (\ref{eq-cal}) converges to 
\[
\begin{cases}
16 \frac{n-1}{n-2} \mathrm{Vol}(S^{n-1}) \frac{\sqrt{\pi}}{2} \prod_{s= 0}^{m} \left( \frac{2s+1}{2} \right)\, (> 0)& \mathrm{if}~n = 2m +1~(m \ge 1), \\
8 \, \frac{n-1}{n-2} \mathrm{Vol}(S^{n-1}) \prod^{m}_{s=1} s\, (> 0) & \mathrm{if}~n = 2m~(m \ge 2).
\end{cases}
\]
Hence, for all sufficiently large $i,$
\[
\begin{split}
\int_{\mathbb{R}^{n}} &R(g_{i})\, d\mathrm{vol}_{g_{i}} \\
&\ge 
\begin{cases}
8 \frac{n-1}{n-2} \mathrm{Vol}(S^{n-1}) \frac{\sqrt{\pi}}{2} \prod_{s= 0}^{m} \left( \frac{2s+1}{2} \right)\, (> 0)& \mathrm{if}~n = 2m +1~(m \ge 1), \\
4 \, \frac{n-1}{n-2} \mathrm{Vol}(S^{n-1}) \prod^{m}_{s=1} s\, (> 0) & \mathrm{if}~n = 2m~(m \ge 2).
\end{cases}
\end{split}
\]

\smallskip
\noindent
Note that $R(g_{i})$ cannot be nonnegative on $\mathbb{R}^{n}$ by the positive mass theorem or the resolution of
Geroch's conjecture on tori \cite{gromov1980spin, schoen1979existence, schoen1979structure}). Indeed, from (\ref{eq-scalar-curvature}),
\[
R(g_{i})(o) = 8 \frac{n(n-1)}{n-2} i^{l} \left( i^{-1+l} + 1 \right)^{-\frac{n+2}{n-2}} ~~> 0,
\]
and
\[
R(g_{i})(x) = -4 \frac{n-1}{n-2} i^{l} \left( i^{-1+l} e^{-\frac{i r^{2}_{0}}{4}} + 1 \right)^{-\frac{n+2}{n-2}} \left( -2 n + i r^{2}_{0} \right)\, e^{-\frac{i r^{2}_{0}}{4}}~~< 0
\]
for any point $x \in \left\{ x \in \mathbb{R}^{n} |~r(x) = \frac{r_{0}}{2} \right\}$ and sufficiently large $i.$
\end{exam}

Next, we will construct a counterexample to Fact \ref{theo-3} in a certain sense, in which each $g_{i}$ is complete.
\begin{exam}[Not $C^{1}$ but $C^{0}$]
\label{exam-C10}
Consider $\left( \mathbb{R}^{n},~g_{i} := u_{i}^{\frac{4}{n-2}} \cdot g_{Eucl} \right)~(n \ge 3,~i = 2,3, \cdots ).$
Here the smooth positive function $u_{i} : \mathbb{R}^{n} \rightarrow \mathbb{R}$ has been defined as
\[
u_{i} = 
\phi \left( i^{-1} \sin(i r^{2}) \right)+ 1.
\]
Here, $\phi : \mathbb{R}^{n} \rightarrow [0,1]$ is a smooth cut-off function
such that $\phi \equiv 1$ on $\overline{B_{r_{0}}} := \{ x \in \mathbb{R}^{n}|~r(x) \le r_{0} \}$ and
$\phi \equiv 0$ outside of the $\varepsilon/2$-neighborhood $\overline{B_{r_{0}}}.$
where $r_{0} > 0$ is an arbitrarily fixed positive constant.
Here, $r : \mathbb{R}^{n} \rightarrow \mathbb{R}_{\ge 0}$ is the Euclidean distance function from the origin $o \in \mathbb{R}^{n}.$
Then, for each $i,$ $(\mathbb{R}^{n}, g_{i})$ is a non-compact smooth Riemannian manifold with
\begin{equation}\label{eq-scal-conformal-2}
  \begin{split}
R(g_{i}) &= u_{i}^{- \frac{n + 2}{n - 2}} \left( - 4 \frac{n-1}{n-2} \Delta_{g_{Eucl}} u_{i} + R(g_{Eucl}) u_{i} \right) \\
&= u_{i}^{- \frac{n + 2}{n - 2}} \left( - 4 \frac{n-1}{n-2} \Delta_{g_{Eucl}} u_{i} \right),
\end{split}
\end{equation}
and
$g_{i} \rightarrow g_{Ecul}$ on $\mathbb{R}^{n}$ in the uniformly $C^{0}$ but not in the $C^{1}$-sense.
On $\overline{B_{r_{0}}},$
\[
\begin{split}
|\nabla u_{i}|^{2} = \sum^{n}_{j=1} \left| \frac{\partial}{\partial x^{j}} i^{-1} \sin(i r^{2}) \right|^{2} 
&= \sum^{n}_{j=1} \left| \frac{\partial r}{\partial x^{j}} \frac{\partial}{\partial r} i^{-1} \sin(i r^{2}) \right|^{2} \\
&= \sum^{n}_{j=1} \left| \frac{x^{j}}{r} (2 r) \cos(i r^{2}) \right|^{2} \\
&= 4 r^{2} \cos^{2}(i r^{2}) \\
&= 2 r^{2} (1+ \cos(2 i r^{2})).    
\end{split}
\]
Note that $R(g_{i})$ is not nonnegative on $\mathbb{R}^{n}.$ Indeed,
for sufficiently large $i$ and $k \in \mathbb{Z}$ such that 
$\overline{B_{r_{0}}} \cap \left\{ x \in \mathbb{R}^{n} |~r(x) = \sqrt{\frac{(2k-1)}{2 i}} \right\} \neq \emptyset,$
we can take a point $x_{i} \in \overline{B_{r_{0}}} \cap \left\{ x \in \mathbb{R}^{n} |~r(x) = \sqrt{\frac{(2k-1)}{2 i}} \right\}.$
Then
\[
R(g_{i}) (x_{i}) \rightarrow
\begin{cases}
\frac{8(n-1)(2k-1) \pi}{n-2} & \mathrm{if}~k~\mathrm{is~odd}, \\
- \frac{8(n-1)(2k-1) \pi}{n-2} & \mathrm{if}~k~\mathrm{is~even}.
\end{cases}
\]
This is checked as follows.
For sufficiently large $i,$ such a point $x_{i}$ is contained in $\overline{B_{r_{0}}}.$
Hence, from the above formula and the choice of the point $x_{i},$ we have
\[
\begin{split}
R(g_{i}) (x_{i}) &=  u_{i}^{- \frac{n + 2}{n - 2}} \left( - 4 \frac{n-1}{n-2} \Delta_{g_{Eucl}} u_{i} \right) \\
&= - 4 \frac{n-1}{n-2} \left( i^{-1} + 1 \right)^{- \frac{n + 2}{n - 2}} \left( 2n \cos(i r(x_{i})^{2}) -4 i r^{2} \sin(i r(x_{i})^{2}) \right) \\
&= (-1)^{k+1} \frac{8(n-1)(2k-1) \pi}{n-2} \left( i^{-1} + 1 \right)^{- \frac{n + 2}{n - 2}}.
\end{split}
\]
Since $\left( i^{-1} + 1 \right)^{- \frac{n + 2}{n - 2}} \rightarrow 1~(i \rightarrow \infty),$
we can observe the desired behavior of the scalar curvature as above.
Moreover, from (\ref{eq-scal-conformal-2}) and the divergence formula, we have
\[
\begin{split}
\int_{\mathbb{R}^{n}} R(g_{i})\, d\mathrm{vol}_{g_{i}}
&= -4 \frac{n-1}{n-2} \int_{\mathbb{R}^{n}} u_{i}^{- \frac{n+2}{n-2}} \Delta_{g_{Eucl}} u_{i} \left( u_{i}^{\frac{2n}{n-2}}\, d\mathrm{vol}_{g_{Eucl}} \right) \\
&= -4 \frac{n-1}{n-2} \int_{\mathbb{R}^{n}} u_{i} \Delta_{g_{Eucl}} u_{i}\, d\mathrm{vol}_{g_{Eucl}} \\
&= 4 \frac{n-1}{n-2} \int_{\left( \overline{B_{r_{0}}} \right)_{\varepsilon}} |\nabla u_{i}|^{2} d\mathrm{vol}_{g_{Eucl}} \\
&\ge 8 \frac{n-1}{n-2} \int_{\overline{B_{r_{0}}}} r^{2} (1+ \cos(2 i r^{2})) d\mathrm{vol}_{g_{Eucl}} \\
&= 8 \frac{n-1}{n-2} \mathrm{Vol}(S^{n-1}) \int_{0}^{r_{0}} r^{2} (1+ \cos(2 i r^{2})) r^{n-1}\, dr.
\end{split}
\]
Here, 
\[
\begin{split}
\int_{0}^{r_{0}} &r^{2} (1+ \cos(2 i r^{2})) r^{n-1} dr
= \left[ \frac{1}{n+2} r^{n+2} \right]_{0}^{r_{0}} + \int_{0}^{r_{0}} r^{n+1} \cos(2 i r^{2})\, dr \\
&= \frac{1}{n+2} r^{n+2}_{0} + \left[ \frac{r^{n} i^{-1}}{4} \sin(2 i r^{2}) \right]^{r_{0}}_{0} + \frac{n i^{-2}}{16} \left[ r^{n-2} \cos(2 i r^{2}) \right]^{r_{0}}_{0} \\
&~~~~~~~~~~~~~~~~~~~~- \frac{n(n-2) i^{-2}}{16} \int^{r_{0}}_{0} r^{n-3} \cos(2 i r^{2})\, dr \\
&\ge  \frac{1}{n+2} r^{n+2}_{0} + \left[ \frac{r^{n} i^{-1}}{4} \sin(2 i r^{2}) \right]^{r_{0}}_{0} + \frac{n i^{-2}}{16} \left[ r^{n-2} \cos(2 i r^{2}) \right]^{r_{0}}_{0} \\
&~~~~~~~~~~~~~~~~~~~~- \frac{n(n-2) i^{-1}}{16} \int^{r_{0}}_{0} r^{n-3}\, dr \\
&=  \frac{1}{n+2} r^{n+2}_{0} + \frac{r_{0}^{n} i^{-1}}{4} \sin(2 i r_{0}^{2}) + \frac{n i^{-2}}{16}\, r_{0}^{n-2} \cos(2 i r_{0}^{2}) - \frac{n i^{-2}}{16} r_{0}^{n-2}.
\end{split}
\]
Since 
\[
\frac{r_{0}^{n} i^{-1}}{4} \sin(2 i r_{0}^{2}) + \frac{n i^{-2}}{16}\, r_{0}^{n-2} \cos(2 i r_{0}^{2}) - \frac{n i^{-2}}{16} r_{0}^{n-2} \rightarrow 0~~\mathrm{as}~i \rightarrow \infty,
\]
there is a sufficiently large $i_{0} = i_{0}(n, r_{0})$ such that for all $i \ge i_{0},$ 
\[
\frac{r_{0}^{n} i^{-1}}{4} \sin(2 i r_{0}^{2}) + \frac{n i^{-2}}{16}\, r_{0}^{n-2} \cos(2 i r_{0}^{2}) - \frac{n i^{-2}}{16} r_{0}^{n-2} 
> - \frac{1}{2(n+2)} r^{n+2}_{0}.
\]
Hence, for all $i \ge i_{0},$
\[
\int_{\mathbb{R}^{n}} R(g_{i})\, d\mathrm{vol}_{g_{i}} > 4 \frac{n-1}{(n+2)(n-2)} \mathrm{Vol}(S^{n-1})\, r_{0}^{n+2} > 0.
\]
\end{exam}
\begin{ques}
    Does the same statement as Fact \ref{theo-1} hold on open manifolds?
    Related to this question and the above example \ref{exam-C10}, are there examples that converge with respect to the $C^{0} \cap W^{1,2}$-topology but do not preserve the lower bound of the total scalar curvature?
\end{ques}

\bigskip
From the Morrey embedding, we have
\[
 W^{1, p} \hookrightarrow C^{0, 1 - \frac{n}{p}} \hookrightarrow C^{0}~~~\mathrm{if}~p > n.
\]
Therefore the same statement of Fact \ref{theo-3} still holds even though one replace $C^{0} \cap W^{1,2}$ with $W^{1, p}~(p > n)$.
On the other hand, in Fact \ref{theo-3}, if we weaken the assumption from $C^{0} \cap W^{1, 2}$ to $C^{0},$ then the same statement does not hold in general.
Indeed, using the same local construction as in the previous example in dimension $\ge 3,$ we can also construct a counterexample on a closed manifold to Fact \ref{theo-3} as follows.
Note that each metric $g_{i}$ in each example below has sign-changing scalar curvature, i.e., for each $i,$ there are some points $x_{i}, y_{i} \in M$ s.t. $R(g_{i})(x_{i}) < 0 < R(g_{i})(y_{i}).$

\begin{exam}[On every closed manifold]
\label{exam-closed}
Consider $\left( M^{n},~g_{i} := u_{i}^{\frac{4}{n-2}} \cdot g_{0} \right)$ $(n \ge 3,~i = 2,3, \cdots ),$
where $M^{n}$ is a closed $n$-manifold and $g_{0}$ is a Riemannian metric on $M.$
Here the smooth positive function $u_{i} : M \rightarrow \mathbb{R}$ has been defined as
\[
u_{i} = 
\phi \left( i^{-1} \sin(i h^{2}) \right)+ 1.
\]
Here, $\phi : M \rightarrow [0,1]$ is a smooth cut-off function
such that $\phi \equiv 1$ on $\overline{B_{r_{0}}}(p) := \{ x \in M |~d_{g_{0}}(p, x) \le r_{0} \}$ and
$\phi \equiv 0$ outside of the $\varepsilon/2$-neighbourhood $\overline{B_{r_{0}}}$ for some point $p \in M$
where $0 < r_{0} < \mathrm{inj}(M, g_{0})$ is a sufficiently small positive constant.
Here, $h :=d_{g_{0}}(\cdot, p) : M \rightarrow \mathbb{R}_{\ge 0}$ is the distance function of $g_{0}$ from the point $p$
and $\mathrm{inj}(M, g_{0})$ is the injectivity radius of $(M, g_{0}).$
Then, for each $i,$ $(M , g_{i})$ is a smooth Riemannian manifold with
\[
R(g_{i}) = u_{i}^{- \frac{n + 2}{n - 2}} \left( - 4 \frac{n-1}{n-2} \Delta_{g_{0}} u_{i} + R(g_{0}) u_{i} \right)
\]
and $g_{i}$ converges to $g_{0}$ on $M$ in the $C^{0}$ but not in the $C^{1}$-sense.
In the same calculation as in the previous example, we have
\[
\begin{split}
\int_{M} &R(g_{i})\, d\mathrm{vol}_{g_{i}} \\
&= -4 \frac{n-1}{n-2} \int_{M} u_{i}^{- \frac{n+2}{n-2}} \left( \Delta_{g_{0}} u_{i} + R_{g_{0}} u_{i} \right)\, u_{i}^{\frac{2n}{n-2}}\, d\mathrm{vol}_{g_{0}} \\
&= -4 \frac{n-1}{n-2} \int_{M} \left( u_{i} \Delta_{g_{0}} u_{i} + R_{g_{0}} u_{i}^{2} \right)\, d\mathrm{vol}_{g_{0}} \\
&= 4 \frac{n-1}{n-2} \int_{\left( \overline{B_{r_{0}}} \right)_{\varepsilon}} |\nabla u_{i}|_{g_{0}}^{2} d\mathrm{vol}_{g_{0}} 
+ \int_{M} R(g_{0}) u_{i}^{2}\, d\mathrm{vol}_{g_{0}} \\
&\ge 8 \frac{n-1}{n-2} \int_{\overline{B_{r_{0}}}} h^{2} (1+ \cos(2 i h^{2})) d\mathrm{vol}_{g_{0}} + \int_{M} R(g_{0}) u_{i}^{2}\, d\mathrm{vol}_{g_{0}} \\
&\ge 8 \frac{n-1}{n-2} \widetilde{C} \int_{0}^{r_{0}} h^{2} (1+ \cos(2 i h^{2})) h^{n-1}\, dh
+ \int_{M} R(g_{0}) u_{i}^{2}\, d\mathrm{vol}_{g_{0}}.
\end{split}
\]
Here, the constant $\widetilde{C}$ depends only on $n$ and $g_{0}.$
Thus, from the observation as in the previous example, there is $i_{0} \in \mathbb{N}$ and a positive constant $C = C(n, g_{0}, r_{0}) > 0$ such that
for all $i \ge i_{0},$
\[
8 \frac{n-1}{n-2} \widetilde{C} \int_{0}^{r_{0}} h^{2} (1+ \cos(2 i h^{2})) h^{n-1}\, dh \ge C > 0.
\]
Therefore, for all $i \ge i_{0},$
\[
\int_{M} R(g_{i})\, d\mathrm{vol}_{g_{i}} \ge C + \int_{M} R(g_{0}) u_{i}^{2}\, d\mathrm{vol}_{g_{0}}.
\]
Moreover, by the definition of $u_{i},$
\[
\left| \int_{M} R(g_{0}) u^{2}_{i}\, d\mathrm{vol}_{g_{0}} - \int_{M} R(g_{0})\, d\mathrm{vol}_{g_{0}} \right|
\le \int_{M} \left| R(g_{0}) \right| \left( 2i^{-1} + i^{-2} \right)\, d\mathrm{vol}_{g_{0}}.
\]
Hence, there is a sufficiently large $i_{1}$ such that for all $i \ge i_{1},$
\[
\left| \int_{M} R(g_{0}) u^{2}_{i}\, d\mathrm{vol}_{g_{0}} - \int_{M} R(g_{0})\, d\mathrm{vol}_{g_{0}} \right| \le \frac{C}{2}.
\]
Thus, for all $i \ge \max \{i_{0}, i_{1} \},$ we have
\[
\int_{M} R(g_{i})\, d\mathrm{vol}_{g_{i}} \ge C + \int_{M} R(g_{0}) u_{i}^{2}\, d\mathrm{vol}_{g_{0}} > \int_{M} R(g_{0})\, d\mathrm{vol}_{g_{0}}.
\]
\end{exam}
Here, we have a question about the regularity of convergence in the assumption of Main Theorem \ref{theo-2}.
\begin{ques}
Fix $a,b,c \in \mathbb{R}$.
Are there any $C^{2}$-metrics $(g_{i})$ and $\Lambda$-Lipschitz ($\Lambda > 0$) functions $(f_{i})$ on a closed $n$-manifold $M^{n}~(n \ge 3)$ satisfying the followings ?
\begin{itemize}
\item $g_{i}$ converges to a $C^{2}$-metric $g$ in the $W^{1, \frac{n^{2}}{2}}$-sense,
\item $f_{i}$ converges to a $\Lambda$-Lipschitz function $f$ in the uniformly $C^{0}$-sense,
\item there is a constant $\kappa$ such that $\int_{M} R_{f_{i}}^{a,b,c}(g_{i})\, e^{-f_{i}}d\mathrm{vol}_{{g}_{i}} \ge \kappa > \int_{M} R_{f}^{a.b.c}(g)\, e^{-f}d\mathrm{vol}_{g}.$ 
\end{itemize}
Or, additionally,
\begin{itemize}
\item there is a point $p_{i} \in M$ for each $i$ such that $R(g_{i})(p_{i}) \rightarrow -\infty$ as $i \rightarrow \infty.$
\end{itemize}
\end{ques}

In the following Example \ref{exam-C21}, we give another counterexample which is similar to the one in Example \ref{exam-C10} ($n \ge 3$).
However, in the following example of dimension $\ge 3,$ the support of $u_{i} - 1$ ($\subset (\mathbb{R}^{n}, d_{g_{Eucl}})$) with the origin $o \in \mathbb{R}^{n}$ converges to $(\mathbb{R}^{n}, g_{Eucl}, o)$ in the pointed Gromov-Hausdorff sense as $i \rightarrow \infty.$
Note that in Example \ref{exam-C10}, the support of $u_{i} -1$ is contained in a fixed compact subset.
Hence, unfortunately, it is not possible to localize this construction directly and construct such a counterexample on a closed manifold as in Example \ref{exam-closed}.
On the other hand, in the two-dimensional example of Example \ref{exam-C21}, the support of $e^{u_{i}} - 1$ is not compact for each $i$ (see the last half of Example \ref{exam-C21}).
\begin{exam}[Not $C^{2}$ but $C^{1}$]
\label{exam-C21}
We will construct an example similar to the one in Example \ref{exam-C10}. However, in this example, the topology of the convergence of the metrics is different.

Consider $\left( \mathbb{R}^{n},~g_{i} := u_{i}^{\frac{4}{n-2}} \cdot g_{Eucl} \right)~(n \ge 3,~i = 2,3, \cdots ).$
Here the smooth positive function $u_{i} : \mathbb{R}^{n} \rightarrow \mathbb{R}$ has been defined as
\[
u_{i} = 
\phi_{i} \left( i^{-2} \sin(i r^{2}) \right)+ 1.
\]
Here, $\phi_{i} : \mathbb{R}^{n} \rightarrow [0,1]$ is a smooth cut-off function
such that $\phi_{i} \equiv 1$ on $\overline{B_{r_{i}}} := \{ x \in \mathbb{R}^{n}|~r(x) \le r_{i} \}$ and
$\phi_{i} \equiv 0$ outside of the $\varepsilon/2$-neighborhood $\overline{B_{r_{i}}}.$
where $r_{i} := i^{\frac{2}{n+2}}.$ Note that $r_{i} \rightarrow \infty$ as $i \rightarrow \infty.$
Here, $r : \mathbb{R}^{n} \rightarrow \mathbb{R}_{\ge 0}$ is the Euclidean distance function from the origin $o.$
Then, for each $i,$ $(\mathbb{R}^{n}, g_{i})$ is a non-compact smooth Riemannian manifold with
\[
\begin{split}
R(g_{i}) &= u_{i}^{- \frac{n + 2}{n - 2}} \left( - 4 \frac{n-1}{n-2} \Delta_{g_{Eucl}} u_{i} + R(g_{Eucl}) u_{i} \right) \\
&= u_{i}^{- \frac{n + 2}{n - 2}} \left( - 4 \frac{n-1}{n-2} \Delta_{g_{Eucl}} u_{i} \right).
\end{split}
\]
Moreover, $g_{i}$ converges to $g_{Ecul}$ on $\mathbb{R}^{n}$ in the uniformly $C^{1}$ but not in the $C^{2}$-sense.
On $\overline{B_{r_{i}}},$
\[
\begin{split}
|\nabla u_{i}|^{2} = \sum^{n}_{j=1} \left| \frac{\partial}{\partial x^{j}} i^{-2} \sin(i r^{2}) \right|^{2} 
&= \sum^{n}_{j=1} \left| \frac{\partial r}{\partial x^{j}} \frac{\partial}{\partial r} i^{-2} \sin(i r^{2}) \right|^{2} \\
&= \sum^{n}_{j=1} \left| \frac{x^{j}}{r} (2 i^{-1} r) \cos(i r^{2}) \right|^{2} \\
&= 4 i^{-2} r^{2} \cos^{2}(i r^{2}) \\
&= 2 i^{-2} r^{2} (1+ \cos(2 i r^{2})).    
\end{split}
\]
Note that $R(g_{i})$ is not nonnegative on $\mathbb{R}^{n}.$ Indeed,
when $i \rightarrow \infty,~R(x)$ oscillates for each $x \in \{ x \in \mathbb{R}^{n} |~r(x) \neq 0 \}.$ 
Indeed, for sufficiently large $i,$ such a point $x$ is contained in $\overline{B_{r_{i}}}.$
Hence, from the above formula,
\[
\begin{split}
R(g_{i}) (x) &=  u_{i}^{- \frac{n + 2}{n - 2}} \left( - 4 \frac{n-1}{n-2} \Delta_{g_{Eucl}} u_{i} \right) \\
&= - 4 \left( \frac{n-1}{n-2} \right) \frac{2n i^{-1} \cos(i r(x)^{2}) -4 r(x)^{2} \sin(i r(x)^{2})}{\left( i^{-2} \sin(i r(x)^{2}) + 1 \right)^{\frac{n + 2}{n - 2}}}.
\end{split}
\]
Since $\left( i^{-1} + 1 \right)^{- \frac{n + 2}{n - 2}} \rightarrow 1$
and
$2n i^{-1} \cos(i r(x)^{2}) \rightarrow 0$
as $i \rightarrow \infty,$
we can easily observe the desired behavior of the scalar curvature.
Moreover, by the divergence formula, 
\[
\begin{split}
\int_{\mathbb{R}^{n}} R(g_{i})\, d\mathrm{vol}_{g_{i}}
&= -4 \frac{n-1}{n-2} \int_{\mathbb{R}^{n}} u_{i}^{- \frac{n+2}{n-2}} \Delta_{g_{Eucl}} u_{i} \left( u_{i}^{\frac{2n}{n-2}}\, d\mathrm{vol}_{g_{Eucl}} \right) \\
&= -4 \frac{n-1}{n-2} \int_{\mathbb{R}^{n}} u_{i} \Delta_{g_{Eucl}} u_{i}\, d\mathrm{vol}_{g_{Eucl}} \\
&= 4 \frac{n-1}{n-2} \int_{\left( \overline{B_{r_{i}}} \right)_{\varepsilon}} |\nabla u_{i}|^{2} d\mathrm{vol}_{g_{Eucl}} \\
&\ge 8 \frac{n-1}{n-2} \int_{\overline{B_{r_{i}}}} r^{2} i^{-2} (1+ \cos(2 i r^{2})) d\mathrm{vol}_{g_{Eucl}} \\
&= 8 \frac{n-1}{n-2} \mathrm{Vol}(S^{n-1}) \int_{0}^{r_{i}} r^{2} i^{-2} (1+ \cos(2 i r^{2})) r^{n-1}\, dr.
\end{split}
\]
Here, 
\[
\begin{split}
\int_{0}^{r_{i}} &i^{-2} r^{2} (1+ \cos(2 i r^{2})) r^{n-1} dr
= \left[ \frac{i^{-2}}{n+2} r^{n+2} \right]_{0}^{r_{i}} + \int_{0}^{r_{i}} i^{-2} r^{n+1} \cos(2 i r^{2})\, dr \\
&= \frac{1}{n+2} + \left[ \frac{r^{n} i^{-3}}{4} \sin(2 i r^{2}) \right]^{r_{i}}_{0} + \frac{n i^{-4}}{16} \left[ r^{n-2} \cos(2 i r^{2}) \right]^{r_{i}}_{0} \\
&~~~~~~~~~~~~~~~~~~~~- \frac{n(n-2) i^{-4}}{16} \int^{r_{i}}_{0} r^{n-3} \cos(2 i r^{2})\, dr \\
&\ge  \frac{1}{n+2} + \left[ \frac{r^{n} i^{-3}}{4} \sin(2 i r^{2}) \right]^{r_{i}}_{0} + \frac{n i^{-4}}{16} \left[ r^{n-2} \cos(2 i r^{2}) \right]^{r_{i}}_{0} \\
&~~~~~~~~~~~~~~~~~~~~- \frac{n(n-2) i^{-4}}{16} \int^{r_{i}}_{0} r^{n-3}\, dr \\
&=  \frac{1}{n+2} + \frac{r_{i}^{n} i^{-3}}{4} \sin(2 i r_{i}^{2}) + \frac{n i^{-4}}{16}\, r_{i}^{n-2} \cos(2 i r_{i}^{2}) - \frac{n i^{-4}}{16} r_{i}^{n-2} \\
&\ge \frac{1}{n+2} + \frac{i^{\frac{2n}{n+2}} i^{-3}}{4} \sin(2 i^{1+ \frac{4}{n+2}}) + \frac{n i^{-4}}{16}\, i^{\frac{2(n-2)}{n+2}} \cos(2 i^{1+ \frac{4}{n+2}}) - \frac{n i^{-4}}{16} i^{\frac{2(n-2)}{n+2}}.
\end{split}
\]
Since 
\[
\frac{i^{\frac{2n}{n+2}} i^{-3}}{4} \sin(2 i^{1+ \frac{4}{n+2}}) + \frac{n i^{-4}}{16}\, i^{\frac{2(n-2)}{n+2}} \cos(2 i^{1+ \frac{4}{n+2}}) - \frac{n i^{-4}}{16} i^{\frac{2(n-2)}{n+2}} \rightarrow 0~~\mathrm{as}~i \rightarrow \infty,
\]
there is a sufficiently large $i_{0} = i_{0}(n)$ such that for all $i \ge i_{0},$
\[
\frac{i^{\frac{2n}{n+2}} i^{-3}}{4} \sin(2 i^{1+ \frac{4}{n+2}}) + \frac{n i^{-4}}{16}\, i^{\frac{2(n-2)}{n+2}} \cos(2 i^{1+ \frac{4}{n+2}}) - \frac{n i^{-4}}{16} i^{\frac{2(n-)}{n+2}}
> - \frac{1}{2(n+2)}.
\]
Hence for all $i \ge i_{0},$
\[
\int_{\mathbb{R}^{n}} R(g_{i})\, d\mathrm{vol}_{g_{i}} > 4 \frac{n-1}{(n+2)(n-2)} \mathrm{Vol}(S^{n-1}) > 0.
\]

\bigskip
Next, we will construct a two-dimensional example.
Consider the smooth function $u_{i}$ on $\mathbb{R}^{2}$ defined by
\[
u_{i} := e^{-ir^{2}} \sin \left( -\frac{i}{2} r^{2} \right)~~(i = 1,2, \cdots ),
\]
where $r(\cdot) := |o - \cdot|$ denotes the Euclidean distance function from the origin $o \in \mathbb{R}^{2}.$
Then $u_{i}$ uniformly converges to the constant function $0$ in the $C^{1}$ topology in $\mathbb{R}^{2}$, but $u_{i}$ does not converge to $0$ in the $C^{2}$ topology in $\mathbb{R}^{2}$.
Hence the sequence of complete metrics $(g_{i} := e^{u_{i}} g_{Eucl})$ on $\mathbb{R}^{2}$ uniformly converges to $g_{Eucl}$ in the $C^{1}$ sense on $\mathbb{R}^{2},$ but $g_{i}$ does not converge to $g_{Eucl}$ in the $C^{2}$ sense on $\mathbb{R}^{2}.$
Set $a := -i, b := \frac{1}{2} a = -\frac{i}{2}.$
Then we can check that
\[
\begin{split}
\Delta u_{i} &= (4a + 4a^{2} r^{2}) e^{ar^{2}} \sin (br^{2}) + (4b \cos r^{2} - 4b^{2} r^{2} \sin br^{2}) e^{ar^{2}} \\
&~~~~~~~~~~+ 8ab r^{2} e^{ar^{2}} \cos br^{2},
\end{split}
\]
and 
\[
\int_{\mathbb{R}^{2}} R(g_{i})\, d\mathrm{vol}_{g_{i}} = - \int_{\mathbb{R}^{2}} \Delta_{g_{Eucl}} u_{i}\, d\mathrm{vol}_{g_{Ecul}}
= -\int^{2\pi}_{0} \int_{0}^{\infty} r \Delta_{g_{Eucl}} u_{i}\, dr\, d\theta.
\]
Moreover,
\begin{itemize}
\item $I := \int^{\infty}_{0} r e^{ar^{2}} \cos br^{2}\, dr = -\frac{a}{2(a^{2} + b^{2})},$

\item $J := \int^{\infty}_{0} r e^{ar^{2}} \sin br^{2}\, dr = -\frac{b}{2(a^{2} + b^{2})},$

\item $K := \int^{\infty}_{0} r^{3} e^{ar^{2}} \cos br^{2}\, dr = -\frac{aI +bJ}{a^{2} + b^{2}},$
\item $L := \int^{\infty}_{0} r^{3} e^{ar^{2}} \sin br^{2}\, dr = -\frac{aJ - bI}{a^{2} + b^{2}}.$
\end{itemize}
Combining these, we obtain that
\[
\begin{split}
\int_{\mathbb{R}^{2}} R(g_{i})\, d\mathrm{vol}_{g_{i}}
&= -2\pi \left( 4b I + 4a J - 4(a^{2} - b^{2}) \frac{aJ - bI}{a^{2} + b^{2}} -8ab \frac{bJ -aI}{a^{2} + b^{2}} \right) \\
&= 2\pi \frac{8a^{3} b}{(a^{2} + b^{2})^{2}} \\
&= \frac{128 \pi}{25} > 0 = \int_{\mathbb{R}^{2}} R(g_{Eucl})\, d\mathrm{vol}_{g_{Ecul}}.
\end{split}
\]
Note that we have used $b = \frac{1}{2} a$ in the third equality.
\end{exam}

\begin{ques}[c.f. \ref{rema-counterexample}]
As we have seen in Example \ref{exam-closed}, in Fact \ref{theo-3}, we cannot weaken the assumptions that the manifold is closed and the convergence is in the sense of $C^{0} \cap W^{1, 2}$ to that the manifold is open and the convergence is in the sense of $C^{0}$ respectively.  
On the other hand, can we weaken the assumptions in Main Theorems \ref{theo-2} and \ref{theo-1} in any sense?
\end{ques}
\begin{rema}
In the above examples, we have constructed these counterexamples by deforming the Euclidean metric locally in a conformal direction.
Then, due to the factors from changes of the volume forms, the total scalar curvatures are uniformly bounded from below by a positive constant.
On the other hand, if we try to investigate similar examples for the weighted total scalar curvature $\int_{M} R^{a,b,c}_{f}(g)\, e^{-f} d\mathrm{vol}_{g}$ (i.e., counterexamples to Main Theorem \ref{theo-2}), we don't know whether or not we can use the same method in the above examples in this situation.
\end{rema}

\appendix
\section{Appendix}\label{section-appendix}
\begin{lemm}[Gronwall's inequality]
\label{lemm-gronwall}
    Let $u, \alpha, \beta : [0, \tau] \rightarrow \mathbb{R}$ be continuous functions. 
    Assume that 
    \begin{equation}\label{eq-gronwall-assumption}
    \frac{\mathrm{d}}{\mathrm{d}t}u(t) \ge \alpha(t) - \beta (t) u(t)
    \end{equation}
    for all $t \in (0, \tau]$.
    Then, 
    \[
    u(t) \ge \left( \exp \int^{t}_{0} \beta(s)\, ds \right)^{-1} \left( u(0) + \int^{t}_{0} \left( \exp \int^{s}_{0} \beta(r)\, dr \right) \alpha(s)\, ds \right)
    \]
    for all $t \in [0, \tau]$.
\end{lemm}
\begin{proof}
    Define a new function
    \[
    v(s) := \exp \int^{s}_{0} \beta(r)\, dr.
    \]
    Then, from the assumption (\ref{eq-gronwall-assumption}),
    \[
    \frac{\mathrm{d}}{\mathrm{d}s}(v(s) u(s)) \ge v(s) \alpha(s).
    \]
    By integration on $[0, t]$, we have
    \[
    v(t) u(t) - v(0) u(0) \ge \int^{t}_{0} v(s) \alpha(s)\, ds
    \]
    for all $t \in [0, \tau]$.
    Since $v(0) = 1$, we obtain the desired assertion by dividing both sides of this inequality.
\end{proof}
\bigskip
\begin{flushleft}
  \textbf{Acknowledgements} The author thanks Prof. Boris Botvinnik for asking the question about Main Theorem \ref{theo-2}.
The author also thanks Prof. Kazuo Akutagawa for giving him the opportunity to visit the University of Oregon from September 25 to October 10, 2022.

   \medskip\noindent
  \textbf{Author Contributions} SH has written the manuscript.

  \medskip\noindent
  \textbf{Funding} The author was supported by JSPS KAKENHI Grant Number 24KJ0153.

  \medskip\noindent
  \textbf{Data availability} Not applicable.
\end{flushleft}

\section*{Declarations}

\begin{flushleft}
  \textbf{Conflict of interest} The author declares that there is no conflict of interest.

  \medskip\noindent
  \textbf{Ethics approval and consent to participate} Not applicable.

  \medskip\noindent
  \textbf{Consent for publication} The author declares the consent for publication.
\end{flushleft}

\end{document}